\documentclass[journal]{IEEEtran}
\usepackage{hyperref}
\usepackage{pdfsync}
\usepackage{color}
\usepackage{amsthm}
\usepackage{amsmath}
\usepackage{graphicx}
\usepackage{caption}
\usepackage{cite}

\usepackage{mathtools}
\mathtoolsset{showonlyrefs}

\usepackage{amsmath,amsfonts,amssymb}
\usepackage{graphicx}%

\usepackage[normalem]{ulem}

\makeatletter
\newcommand*{\T}{%
  {\mathpalette\@T{}}%
}
\newcommand*{\@T}[2]{%
  \raisebox{\depth}{$\m@th#1\intercal$}%
}
\makeatother

\newtheorem{theorem}{Theorem}

\newtheorem{definition}[theorem]{Definition}
\newtheorem{lemma}[theorem]{Lemma}

\newtheorem{remark}[theorem]{Remark}
\newtheorem{example}[theorem]{Example}

\newtheorem{assumption}{Assumption}

\newtheorem{newassumption}{Assumption}

\def\vq{{\bf q}}   
\def\vu{{\bf u}}   \def\vx{{\bf x}}
\def\vy{{\bf y}} \def\vz{{\bf z}}

  \def\calC{\mathcal{C}}

 \def\calN{\mathcal{N}} 
  
\def\calS{\mathcal{S}}

\newcommand{\e}{\varepsilon}

\newcommand{\dt}{\, dt}
\newcommand{\ds}{\, ds}

\newcommand{\R}{\mathbb{R}}

\newcommand{\vxperp}{\vx_{nc}}

\newcommand{\ddt}{\frac{d}{dt}}

\newcommand{\dtau}{\,d\tau}
\newcommand{\myspan}{\textup{span}}

\newcommand{\C}{\mathcal{C}}

\newcommand{\diag}{\text{diag}}
\newcommand{\ns}{{n_s}}

\newcommand{\Hess}{\nabla^2}

\definecolor{darkgreen}{rgb}{0.1,0.6,0.1}

\title{\LARGE Distributed Gradient Flow:\\  Nonsmoothness, Nonconvexity, and Saddle Point Evasion}

\author{Brian Swenson,\IEEEauthorrefmark{2} Ryan Murray,\IEEEauthorrefmark{4} H. Vincent Poor,\IEEEauthorrefmark{2} and Soummya Kar\IEEEauthorrefmark{3}\thanks{The work of B. Swenson and H. V. Poor was partially supported by the
Air Force Office of Scientific Research under MURI Grant FA9550-18-1-0502.}}

\begin{document}

\maketitle

\renewcommand{\thefootnote}{\fnsymbol{footnote}}

\footnotetext[2]{Department of Electrical Engineering, Princeton University, Princeton, NJ (bswenson@princeton.edu, poor@princeton.edu).}
\footnotetext[4]{Department Mathematical Sciences, North Carolina State University, Raleigh, NC (rwmurray@ncsu.edu).}
\footnotetext[3]{Department of Electrical and Computer Engineering, Carnegie Mellon University, Pittsburgh, PA (soummyak@andrew.cmu.edu).}

\renewcommand{\thefootnote}{\arabic{footnote}}

\thispagestyle{empty}
\begin{abstract}
The paper considers \emph{distributed} gradient flow (DGF) for multi-agent nonconvex optimization. DGF is a continuous-time approximation of distributed gradient descent that is often easier to study than its discrete-time counterpart. The paper has two main contributions.  First, the paper considers optimization of nonsmooth, nonconvex objective functions. It is shown that DGF converges to critical points in this setting. The paper then considers the problem of avoiding saddle points. It is shown that if agents' objective functions are assumed to be smooth and nonconvex, then DGF can only converge to a saddle point from a zero-measure  set of initial conditions. To establish this result, the paper proves a stable manifold theorem for DGF, which is a fundamental contribution of independent interest. In a companion paper, analogous results are derived for discrete-time algorithms.
\end{abstract}



\section{Introduction} \label{sec:intro}
In this paper we are interested in multi-agent algorithms for optimizing the function
\begin{equation} \label{eqn:f-def}
f(x)\coloneqq \sum_{n=1}^N f_n(x),
\end{equation}
where $N$ denotes the number of agents, and $f_n:\R^d\to\R$ represents a private function available only to agent $n$. Agents are assumed to be equipped with some communication graph that may be used to exchange information with neighboring agents. 
We will consider the behavior of \emph{distributed gradient flow} (DGF)---a multi-agent version of classical (centralized) gradient flow, formally defined in \eqref{dynamics_CT} below---for optimizing \eqref{eqn:f-def} when each $f_n$ is permitted to be nonconvex and possibly nonsmooth. 

Problems of the form \eqref{eqn:f-def}, particularly with nonconvex objectives, arise in numerous applications \cite{di2016next,zhu2012approximate,hong2017prox}. Of particular recent interest, problems of this form are ubiquitous in distributed machine learning and training of deep neural networks \cite{jain2017non,lian2017can}. In practice, first-order methods such as (discrete-time) gradient descent, and (continuous-time) gradient flow are indispensable tools in handling such problems. In large-scale multi-agent settings where information is not centrally available, it is necessary to utilize distributed variants of these processes. 


The paper has two main contributions. First, we consider convergence to critical points of \eqref{eqn:f-def} when objectives are nonconvex and nonsmooth.\footnote{Through the entire paper we allow for nonconvex objectives. However, in later results we will make some smoothness assumptions.}  Nonsmooth objectives frequently arise in practice---for example,  $\ell_1$ regularization is commonly employed to avoid overfitting, and in the context of neural networks, nonsmooth ReLU activation functions are often preferred by practitioners (which in turn lead to nonsmooth nonconvex objective functions) \cite{goodfellow2016deep}.
The first main contribution will be to show that DGF converges to critical points of \eqref{eqn:f-def} in this setting (Theorem \ref{thrm:cont-conv-cp}). 
Formally, the only assumptions we will make on the objective for this result are Assumptions \ref{a:Lip-fn}, \ref{a:coercive}, \ref{a:critical-points}, and \ref{a:chain-rule} below.
These assumptions are quite broad---among other things, they encompass a wide range of data science applications, including popular (nonsmooth) neural network architectures (cf. \cite{davis2020stochastic}). 
To the best of our knowledge, these are the weakest assumptions on the objective function for which DGF, or more generally, any distributed first-order optimization process is currently known to converge to critical points for nonsmooth, nonconvex objectives. A more detailed discussion of related work can be found in Section \ref{sec:lit-review}.


In applications of nonconvex optimization, it is often sufficient to compute local minima. Up to this point, we have only discussed convergence to critical points, which allows for the possibility of convergence to a saddle point (rather than a local minimum). Characterizing the behavior of optimization dynamics near saddle points is a challenging issue---a serious shortcoming of  current literature on distributed first-order algorithms is that most results can only ensure convergence to critical points.
Our second main contribution will be to show that convergence to saddle points of \eqref{eqn:f-def} is ``atypical'' behavior for DGF. In particular, we will see that if we assume a degree of smoothness near saddle points, we can establish a stable-manifold theorem for DGF (Theorem \ref{thrm:non-convergence-saddles}). The stable-manifold theorem for DGF is a powerful result with many important consequences. A simple and immediate consequence is that, if functions are assumed to be globally smooth,  then saddle points can only be reached from a zero-measure set of initial conditions (Theorem \ref{thrm:ae-convergence})---stated in other words, if a DGF process is randomly initialized, then the probability of converging to a saddle point is zero.\footnote{Global smoothness is not required to obtain nonconvergence to saddle points. However, it simplifies the discussion, as pathological cases arise when objective functions lack global smoothness. A more nuanced discussion can be found above Theorem \ref{thrm:ae-convergence}.
Also, here we implicitly assume a random initialization with distribution that is absolutely continuous with respect to the Lebesgue measure.}

The classical stable-manifold theorem is a canonical result from dynamical systems theory that characterizes the behavior of autonomous nonlinear systems near hyperbolic equilibrium points \cite{chicone2006ordinary}.\footnote{An equilibrium point is said to be hyperbolic if the Jacobian of the vector field is invertible at the equilibrium point.} 
Informally, the classical stable-manifold theorem tells us the following for centralized first-order optimization dynamics: \emph{Typical saddle points can only be reached from some smooth low-dimensional (zero-measure) surface.}

It is, of course, a well established fact that centralized gradient flows do not typically converge to saddle points, and this fact is a direct consequence of the classical stable-manifold theorem (see Section \ref{sec:lit-review} for references). Unfortunately, in distributed settings the classical stable-manifold theorem is not generally applicable. Hence, our understanding of saddle point nonconvergence in these settings is far less clear. This paper seeks to address this issue by establishing a stable-manifold theorem for DGF. 

We emphasize that, in order to show convergence to critical points (contribution 1), we will not require functions to be smooth. However, to establish nonconvergence to saddle points and the stable-manifold theorem for DGF (contribution 2) we \emph{will} require at least local smoothness near the saddle point. (Intuitively, linearization lies at the heart of the stable-manifold theorem, and it is not clear how to linearize without smoothness.)

In the following section we formally present the main results of the paper. 
\subsection{Setup and Main Results} \label{sec:main-results}
Throughout the paper we will make the following assumption. 
\begin{assumption} \label{a:Lip-fn}
$f_n$ is locally Lipschitz continuous. 
\end{assumption}
Note that while we have not assumed $f_n$ to be differentiable, under Assumption \ref{a:Lip-fn}, 
the derivative of $f_n$ exists almost everywhere. This is a consequence of Rademacher's theorem \cite{evans2015measure}. 
In order to define a distributed gradient-descent process for \eqref{eqn:f-def} satisfying Assumption \ref{a:Lip-fn}, we will consider the following notion of a generalized gradient 
\cite{clarke2008nonsmooth}.
\begin{definition} \label{def:gen-grad}
Given a locally Lipschitz continuous function $g:\R^m\to\R$, the generalized gradient of $g$ is given by 
$$
\partial g(x) := \text{co}\left\{\lim_{i \to \infty} \nabla g(x_i): x_i\to x,\, \nabla g(x_i) \text{ exists}  \right\},
$$
where $\nabla g(x)$ is the classical gradient of $g$ and  $\text{co}\{\cdot\}$ indicates the convex hull.
\end{definition}
When $g$ is locally Lipschitz, $\partial g(x)$ is a nonempty compact convex set for all $x\in\R^d$ \cite{clarke2008nonsmooth}. 
If $g$ is continuously differentiable, then $\partial g$ is a singleton and coincides with the usual notion of the gradient. If $g$ is convex, then $\partial g$ coincides with the subgradient of $g$. Further discussion of generalized gradients in the context of control and discontinuous systems can be found in \cite{cortes2008discontinuous}.

We will assume that agents are endowed with some communication graph $G=(V,E)$ over which they may exchange information with neighboring agents. 
Here, the set of vertices $V=\{1,\ldots,N\}$ represents the set of agents and an edge $(i,j)\in E$ between vertices represents the ability of two agents to exchange information. We will assume the following.
\begin{assumption} \label{a:G-connected-undirected}
The graph $G=(V,E)$ is undirected, unweighted, and connected.
\end{assumption}

Let $\vx_n(t)$ denote the state of agent $n$ at time $t$---this may be thought of as an estimate of an optimizer of \eqref{eqn:f-def} held by agent $n$ at time $t$. 
The DGF process we study in this paper is given by\footnote{Because we consider gradient descent with respect to the generalized gradient, which can be a set, we must consider DGF as a differential \emph{inclusion} rather than an ordinary differential \emph{equation} (ODE). We will consider solutions to differential inclusions in the sense given after \eqref{eqn:generic-ode} below. A primer on differential inclusions can be found in \cite{cortes2008discontinuous} and a more detailed treatment in \cite{aubin2012differential}.}
\begin{equation} \label{dynamics_CT}
\dot \vx_n(t) \in \beta_t \sum_{\ell \in \Omega_n} (\vx_\ell(t) -\vx_n(t)) - \alpha_t \partial f_n(\vx_n(t)),
\end{equation}
where $\alpha_t$ and $\beta_t$ are scalar weight parameters and $\Omega_n$ is the set of neighbors of agent $n$ in the graph $G$. 
The update in \eqref{dynamics_CT} is a continuous-time generalized gradient version of consensus+innovations~\cite{kar2013distributed} and the related class of diffusion~\cite{chen2012diffusion} and distributed gradient descent (DGD)~\cite{nedic2009distributed} processes for distributed optimization. Note that when each $f_n$ is convex, this reduces to a distributed subgradient-descent process, and when each $f_n$ is continuously differentiable, this becomes a standard ODE. 
We emphasize that under Assumption \ref{a:Lip-fn}, the differential inclusion \eqref{dynamics_CT} is well posed since $\partial f_n(x)$ is nonempty, compact and convex \cite{aubin2012differential}. We also emphasize that the process is \emph{distributed} since the dynamics of agent $n$ only depend on locally available information. 

The process \eqref{dynamics_CT} may be intuitively interpreted as follows. The first term on the right-hand side of \eqref{dynamics_CT} is a \emph{consensus term} that draws agents' states closer together, while the second term is a \emph{descent} term that encourages agents to descend their private objective function. In particular, note that if we set $\partial f_n\equiv 0$, then \eqref{dynamics_CT} reduces to a standard continuous-time consensus algorithm \cite{dimakis2010gossip,jadbabaie2003coordination}. 

The first main result of this paper is that under the dynamics \eqref{dynamics_CT}, agents attain consensus and converge to the set of critical points of $f$. Given that $f$ may be nonsmooth, the notion of a critical point is defined as follows \cite{clarke2008nonsmooth}. 
\begin{definition}
We say that $x\in\R^d$ is a critical point of $f$ if $0\in \partial f(x)$.
\end{definition}
Note, of course, that if $f$ is smooth, this generalizes the classical case where a critical point satisfies $\nabla f(x) = 0$, and if $f$ is convex, then this reduces to the standard first-order optimality condition for the subgradient of a convex function. 

To ensure convergence to critical points, we will make a few additional assumptions. First, we will assume that agents' private functions are coercive in the following sense.
\begin{assumption} \label{a:coercive}
$f_n$ is coercive, i.e., $f_n(x)\to \infty$ as $\|x\|\to\infty$.
\end{assumption}
This assumption is relatively weak in the sense that it need only hold asymptotically and does impose any constraints on the rate at which $f_n(x)\to\infty$. 
Under Assumptions \ref{a:Lip-fn} and \ref{a:coercive}, the set of critical points of $f$ is nonempty. 

Next, we assume that the set of \emph{critical values} (the image of the set of critical points) is a ``small'' set. We recall that a set $S\subset \R$ is said to be dense in $\R$ if for each point $x\in \R$ there exists a sequence in $S$ converging to $x$. 
\begin{assumption} \label{a:critical-points}
Let $\text{CP}_f\subset \R^d$ denote the set of critical points of $f$. The set $\R\backslash f(\text{CP}_f)\subset \R$ is a dense set in $\R$.  
\end{assumption}
Note that $f(\text{CP}_f)$ is the set of critical values of $f$, so the assumption stipulates that the set of non-critical values of $f$ is dense in $\R$.
This assumption is relatively weak, and is standard in stochastic approximation literature  \cite{benaim2005stochastic,davis2020stochastic,bianchi2012convergence}. The assumption is satisfied if $f(\text{CP}_f)$ is a zero measure set. Thus, for example, by the well-known theorem of Sard \cite{hirsch2012differential}, the assumption holds whenever $f$ is $d$-times continuously differentiable. The assumption also holds in a wide range of other circumstances of practical interest involving nonsmooth objective functions \cite{davis2020stochastic}.

In the context of smooth optimization, it is trivial to see that if $f$ is smooth and $\vx(t)$ is a gradient flow trajectory, then
\begin{align} \label{eq:descent-smooth}
\ddt f(\vx(t)) & = \langle \nabla f(\vx(t)), \ddt \vx(t)\rangle\\
& = -\|\nabla f(\vx(t))\|^2,
\end{align}
where the first equality follows from the chain rule. This relationship makes clear the critical fact that $f(\vx(t))$ decreases along the trajectory of $\vx(t)$ unless at a critical point. 

In the context of nonsmooth optimization, the key relationship \eqref{eq:descent-smooth} is no longer obvious or trivial. 
To ensure that such a property holds, we must make the following assumption. 
\begin{assumption}[Chain rule] \label{a:chain-rule}
For any absolutely continuous function $\vx:[0,\infty)\to \R^{d}$,
$f$ satisfies the chain rule
$$
\ddt f(\vx(t)) = \langle v, \ddt \vx(t) \rangle, 
$$
for some $v\in \partial f(\vx(t))$, and almost all $t\geq 0$. 
\end{assumption}
This assumption is quite broad, and examples where the assumption fails to hold are typically pathological \cite{daniilidis2020pathological}. 
The problem of identifying explicit function classes for which this assumption holds was studied in 
\cite{davis2020stochastic,drusvyatskiy2015curves} where it was shown that the assumption holds for a broad class of functions (namely, those that are subdifferentiably regular or Whitney stratifiable \cite[Sec. 5]{davis2020stochastic}) that includes popular nonsmooth deep learning architectures as a special case.

Finally, to simplify the analysis, we will assume that the weight parameters $\alpha_t$ and $\beta_t$ take the following form.\footnote{To emphasize that $\alpha_t$ and $\beta_t$ are scaling parameters, and to reduce notational clutter, we have placed the time argument for these in subscripts.}
\begin{assumption} \label{a:step-size-CT}
$\alpha_t = \Theta(t^{-\tau_\alpha})$ and $\beta_t = \Theta( t^{-\tau_\beta})$, with $0\leq \tau_\beta < \tau_\alpha \leq 1$.
\end{assumption}
We note that in the above assumption we use the notation $g(t) = \Theta(h(t))$ to indicate that for some constants $c_1,c_2>0$ we have $c_1 h(t) \leq g(t) \leq c_2h(t)$ for all $t\geq 0$ sufficiently large.

The first main result of the paper is the following, which states that agents reach asymptotic consensus and converge to critical points of \eqref{eqn:f-def}.
\begin{theorem}[Convergence to Critical Points] \label{thrm:cont-conv-cp}
Suppose $\{\vx_n(t)\}_{n=1}^N$ is a solution to \eqref{dynamics_CT} with arbitrary initial condition and suppose that Assumptions \ref{a:Lip-fn}--\ref{a:step-size-CT} hold.  Then for each $n=1,\ldots,N$ we have
\begin{enumerate}
  \item [(i)] $\lim_{t\to\infty} \|\vx_n(t) - \vx_\ell(t)\| = 0$, for all $\ell = 1,\ldots,N$.
  \item [(ii)] $\vx_n(t)$ converges to the set of critical points of $f$.
\end{enumerate}
\end{theorem}

Next, we consider the problem of avoiding saddle points. We will approach this problem by establishing a stable-manifold theorem for DGF.
Up to now, we have allowed for functions with discontinuous gradients and shown convergence to critical points. However, in order to understand nonconvergence to saddle points and establish a stable-manifold theorem for DGF we will make some assumptions about the smoothness of agents' functions. 

We say that $x^*\in \R^d$ is a saddle point of $f$ if $0\in\partial f(x)$ and $x^*$ is neither a local maximum or minimum.
Formally, given a saddle point $x^*$ we will assume the following. 
\begin{assumption} \label{a:C2}
Each $f_n$ is twice continuously differentiable in a neighborhood of $x^*$. 
\end{assumption}
We note that this assumption allows for applications where the objective function may be nonsmooth, 
so long as the saddle point of interest does not occur precisely at a point of gradient discontinuity.

Under Assumption \ref{a:C2}, we will consider saddle points satisfying the following notion of regularity, where we use $\nabla^2 f(x)$ to denote the Hessian of $f$ at $x$. 
\begin{definition} [Nondegenerate or Regular Saddle Point] \label{def:strict saddle}
A saddle point $x^*$ of $f$ will be said to be \emph{nondegenerate} (or \emph{regular}) if the Hessian $\Hess f(x^*)$ is nonsingular.
\end{definition}
The term nondegenerate is standard for this concept in optimization. However, since we will deal with \emph{nonconvergence} to these points, we will generally prefer to use the term ``regular'' to avoid frequent use of double negatives.

We will also require the following assumption, which is quite mild but somewhat technical.
\begin{assumption}[Continuity of Eigenvectors] \label{a:eigvec-continuity}
Suppose $x^*\in\R^d$ is a saddle point of \eqref{eqn:f-def}.
For each $n$, the eigenvectors of $\Hess f_n(x)$ are continuous at $x^*$ in the sense that, for each $x$ in a neighborhood of $x^*$, there exists an orthonormal matrix $U_n(x)$ that diagonalizes $\Hess f_n(x)$ such that $x\mapsto U_n(x)$ is continuous at $x^*$.
\end{assumption}
This assumption is required to rule out certain pathological cases that can arise in the distributed setting.
The assumption is mild and should be satisfied by most functions encountered in practice. (See Example \ref{example-eigval-problem} and related discussion below.) The assumption is guaranteed to hold if each $f_n$ is analytic or if, for each $n$, the Hessian of $f_n$ has no repeated eigenvalues \cite{katoBook}.

Our second main result, stated next, establishes the existence of stable manifolds for DGF. Informally, the theorem states that, in a neighborhood of a regular saddle point, a DGF process $\vx(t) = (\vx_n(t))_{n=1}^N$ can only converge to the saddle point if it is initialized on some special low-dimensional surface.
\begin{theorem} [Stable-Manifold Theorem for DGF] \label{thrm:non-convergence-saddles}
Suppose that $x^*\in\R^{d}$ is a regular saddle point of $f$ and Assumptions \ref{a:G-connected-undirected} and
\ref{a:step-size-CT}--\ref{a:eigvec-continuity}
are satisfied. 
Let $\tilde x = (x^*,\ldots,x^*)\in \R^{Nd}$ be the $N$-fold repetition of $x^*$
Let $q$ denote the number of negative eigenvalues of $\nabla^2 f(x^*)$. Then there exists a neighborhood $\calN\subset\R^{Nd}$ containing $\tilde x$ such that the following holds: For any $t_0\in \R$, let $S_{t_0}$ denote the set of all $x_0\in\calN$ such that $\vx(t)\to \tilde x$ when $\vx(t_0) = x_0$. Then $\calS_{t_0}$ is a smooth (continuously differentiable) $(Nd-q)$-dimensional manifold. 
\end{theorem}

In the above theorem, when we say that $\calS_{t_0}$ is a $C^1$ manifold with dimension $Nd-q$ we mean that $\calS_{t_0}$ is the graph of a $C^1$ function over a $(Nd-q)$-dimensional domain.
We note that in classical settings, the stable manifold does not depend on time. However, because DGF is a nonautonomous system (since $\alpha_t$ and $\beta_t$ are both time-varying), the stable manifold here is time-dependent.

Because we have only assumed local smoothness, the stable manifold theorem above is a local result. In particular, given a regular saddle point $x^*$, it immediately implies that for  almost all initializations in a neighborhood of the saddle point, DGF does not converge to $x^*$. 
However, a challenging aspect of discontinuous dynamical systems such as \eqref{dynamics_CT} is that they can concentrate sets with positive volume into zero measure sets in finite time \cite{cortes2008discontinuous}. Thus, when $f_n$ is nonsmooth, we cannot claim in general that, as a consequence of Theorem \ref{thrm:non-convergence-saddles}, the set of initial conditions in all of $\R^{Nd}$ such that $\vx_n(t)\to x^*$ for some $n$, has measure zero.\footnote{This is because it could occur that the right hand side of \eqref{dynamics_CT} concentrates precisely into the stable manifold of $x^*$ in finite time. In general, we expect that this behavior is pathological for many functions of interest. However, a detailed treatment of this issue is beyond the scope of this paper. We also note that because the stable manifold is inherently \emph{unstable}, this issue can be sidestepped by adding noise to the optimization process. For example, using the stable-manifold theorem from this paper, in \cite{SMKP2020saddles} it is shown that distributed stochastic gradient descent (D-SGD) avoids saddle points with probability 1, regardless of initialization.}
However, if objective functions are globally smooth, then we can say more, as stated in the following theorem. 
\begin{theorem} \label{thrm:ae-convergence}
Suppose that Assumptions \ref{a:G-connected-undirected}--\ref{a:eigvec-continuity} hold and, moreover, each $f_n$ is (globally) $C^2$. Let $x^*$ be a regular saddle point of $f$. Then for each time $t_0$, the set of initial conditions in $\R^{Nd}$ from which $\vx_n(t)$ converges to $x^*$ for some (then every) agent $n$ is a Lebesgue-measure-zero set.  
\end{theorem}
Note that, by Theorem \ref{thrm:cont-conv-cp}, agents achieve consensus under the assumptions in the previous theorem. Thus, if $\vx_n(t)\to x^*$ for some $n$, then this occurs for every $n$. 
Theorem \ref{thrm:ae-convergence} follows from Theorem \ref{thrm:non-convergence-saddles} and the observation that 
the gradient field of a $C^2$ function has bounded divergence. In particular, a simple application of the classical divergence theorem (or Gauss-Green theorem \cite{evans2015measure}) shows that the system cannot concentrate a set of positive measure into a zero measure set in finite time (see, e.g., proof of Proposition 21 in \cite{swenson2018best}). 

\bigskip
\noindent \textbf{Organization}. The remainder of the paper is organized as follows. Section \ref{sec:lit-review} briefly reviews related literature and Section \ref{sec:prelims} sets up notation to be used in the proofs. In order to simplify notation and make proofs more transparent, it will be helpful to consider distributed optimization of \eqref{eqn:f-def} as a special case of a general subspace-constrained optimization problem. Section \ref{sec:general-setup} sets up the general optimization problem that will be used to prove the main results. Section \ref{sec:conv-to-CP} shows convergence to critical points. Section \ref{sec:CT-stable-manifold} proves the stable-manifold theorem for DGF and presents an illustrative example discussing computation of the stable manifold. Finally, Section \ref{sec:conclusions} concludes the paper.

\subsection{Literature Review} \label{sec:lit-review}
Algorithms for distributed optimization with convex cost functions have been studied extensively in the literature. 
While a complete survey of this topic is beyond the scope of the paper, we note that key issues which have been addressed in this context include optimization over time-varying  and directed communication networks \cite{nedic2014distributed,gharesifard2013distributed};
constrained optimization  \cite{nedic2010constrained,zeng2016distributed,zhu2011distributed}; 
convergence rate analysis \cite{jakovetic2014fast,nedic2017achieving}; and optimization of
nonsmooth objectives \cite{nedic2009distributed,duchi2011dual}\cite{liang2017distributed}.
In contrast, in this paper we consider optimization of nonconvex and nonsmooth objective functions. In order to focus our attention squarely on the challenging issues that arise from these assumptions, we restrict our attention to the relatively simple setting of unconstrained optimization over a time-invariant undirected graph. 

The fact that (centralized) continuous-time gradient flows do not converge to saddle points follows from the classical stable-manifold theorem \cite{chicone2006ordinary}. Nonconvergence to saddle points for discrete-time gradient algorithms has been a subject of recent interest \cite{lee2016gradient,lee2019first}---nonconvergence in this setting follows from the stable-manifold theorem for discrete-time dynamical systems \cite{shub2013global}. A related line of recent research has investigated the issue of escaping from saddle points in centralized settings \cite{jin2019nonconvex,jin2017escape,du2017gradient,murray2017revisiting}.

Distributed nonconvex optimization has recently become the subject of intensive research attention. Pioneering early work on this topic can be found in \cite{bianchi2012convergence} which studied a projected variant of DGD for constrained nonconvex optimization and demonstrated convergence to KKT points. The present paper is closely related to \cite{bianchi2012convergence} in that we study the continuous flow underlying DGD and we prove convergence to critical points using techniques from the theory of stochastic approximation and perturbed differential inclusions.
However, our work work differs from \cite{bianchi2012convergence} in significant ways, e.g., we allow for nonsmooth functions and we study the issues of saddle point nonconvergence and existence of stable manifolds for DGF. 

More recent works including  
\cite{hong2017prox,di2016next,sun2016distributed,scutari2019distributed,tatarenko2017non,tian2018asy,sun2019distributed,wai2017decentralized} have addressed various issues related to obtaining convergence to critical points, including dealing with directed graphs, time-varying graphs, and nonsmooth regularizers. The recent work \cite{daneshmand2018second,daneshmand2018second-b} studied the problem of avoiding saddle points with discrete-time DGD with constant step size. Using the classical (discrete-time) stable-manifold theorem, it was shown that DGD with sufficiently small step sizes avoids saddle points and converges to the neighborhood of local minima. References \cite{vlaski2019distributed1,vlaski2019distributed2} study a diffusion adaptation variant of gradient descent and show that under appropriate noise assumptions it is able to escape from saddle points in polynomial time. 
In addition to the fact that we study nonsmooth functions and study the underlying differential inclusions, our work differs from these in that we obtain convergence to consensus and critical points and explicitly characterize the stable manifold associated with saddle points. 
Moreover, because we study the continuous-time flow, the results derived in this paper can be used to approximate discrete-time diminishing-step-size versions of DGD, obtaining convergence in the presence of noise \cite{SMKP2020saddles}.
In another related line of research, annealing based methods for distributed global optimization in nonconvex problems are considered in \cite{swenson2019CDC,swenson2019ICASSP}. While these methods achieve global convergence guarantees, they require careful tuning of the annealing schedule and convergence can be slow in some applications. 

Limited research has been conducted on the topic of distributed optimization when objectives are both nonsmooth and nonconvex. References \cite{di2016next} and \cite{scutari2019distributed} consider convergence to critical points when the distributed objective is the sum of a smooth nonconvex component and a nonsmooth convex component  (or difference-of-convex with smooth convex part). 
In contrast, here we obtain convergence of DGF to critical points under broad assumptions, analogous to the state-of-the art guarantees for centralized (discrete-time) first-order methods found in \cite{davis2020stochastic}. Among other things, these assumptions handle neural networks with nonsmooth activation functions and $\ell_1$ or $\ell_2$ regularization.  

We remark that a preliminary conference version of this paper appeared in \cite{swenson2019allerton}. Most significantly, the present paper differs from \cite{swenson2019allerton} in that it handles nonsmooth objective functions and proves smoothness of the stable manifold (which is required to obtain that the manifold is a measure-zero set). We also note that the present paper fills a gap in the proof of Theorem 1 and 2 in \cite{swenson2019allerton} which requires Assumptions \ref{a:critical-points} and \ref{a:eigvec-continuity}.

As an illustration of the practical applicability of the results derived in this paper, in a related work \cite{SMKP2020saddles} the stable-manifold theorem for DGF (Theorem \ref{thrm:non-convergence-saddles} above) is used to study discrete-time distributed stochastic gradient descent (D-SGD). In particular, in \cite{SMKP2020saddles} it is shown that, regardless of initialization, with probability 1, D-SGD does not converge to regular saddle points. The stable-manifold theorem from this paper plays a critical role in deriving that result. 

\subsection{Notation} \label{sec:prelims}
We say that $g\in C^r(\R^m; \R^n)$, for integer $r\geq 1$, if $g:\R^m\to\R^n$ is $r$-times continuously differentiable. When the domain and codomain are clear from the context, we simply use the shorthand $g\in C^r$ or say $g$ is $C^r$.
If $g$ is $C^1$, we use the notation $D[g,x]$ to denote the derivative of $g$ at the point $x$. 
In the case that $g:\R^n\to\R$ is $C^2$, we often use the standard notation $\nabla g$ and $\Hess g$ to refer to the gradient and Hessian of $g$ respectively.

We will use $\|\cdot\|$ to denote the standard Euclidean norm.  Given a set $S\subset \R^d$ and point $x\in \R^d$, we let $d(x,S) := \inf_{y\in S} \|x-y\|$ and let $B_\delta(S) := \{x: d(x,S)<\delta\}$. When we say $x(k)\to S$ as $k\to\infty$, we mean that $\lim_{k\to\infty} d(x(k),S)=0$. Given $a,b\in\R$, $a\wedge b$ is the minimum of $a$ and $b$. $A \otimes B$ indicates the Kronecker product of matrices $A$ and $B$ of compatible dimension. Given a matrix $A\in \R^{d\times d}$, $\diag(A)$ is the $d$-dimensional vector containing the diagonal entries of $A$. In an abuse of notation, given a vector $v\in \R^d$, we also use $\diag(v)$ to denote the $d\times d$ diagonal matrix with entries of $v$ on the diagonal. 

Given a graph $G=(V,E)$, the set of vertices $V=\{1,\ldots,N\}$ will be used to denote the set of agents and an edge $(i,j)\in E$ will denote the ability of two agents to exchange information. In this paper we will assume $G$ is undirected, meaning that $(i,j) \in E$ implies that $(j,i) \in E$. We let $\Omega_n$ denote the set of neighbors of agent $n$, namely $\Omega_n = \{i \in 1 \dots N : i \neq n, (i,n) \in E\}$, and we let $d_n =|\Omega_n|$.
The graph Laplacian is given by the $N\times N$ matrix $L = D-A$, where $D=\diag(d_1,\ldots,d_N)$ is the degree matrix and $A=(a_{ij})$ is the adjancency matrix defined by $a_{ij} = 1$ if $(i,j)\in E$ and $a_{ij}=0$ otherwise. Further details on spectral graph theory can be found in \cite{chung1997spectral}.

Suppose that $F:\R^m\to\R$ is locally Lipschitz, and consider the differential inclusion
\begin{equation} \label{eqn:generic-ode}
\dot \vx \in \partial F(\vx,t),
\end{equation}
where $\vx:\R\to\R^m$ and $\dot \vx$ denotes $\ddt \vx(t)$. We say $\vx$ is a solution to \eqref{eqn:generic-ode} with initial condition $x_0$ at time $t_0$ if $\vx$ is absolutely continuous and, satisfies $\vx(t_0) = x_0$, and satisfies \eqref{eqn:generic-ode} for almost all $t\geq t_0$.

The generalized gradient (Definition \ref{def:gen-grad}) is known to be upper semicontinuous when the function in question is locally Lipschitz \cite{clarke2008nonsmooth,cortes2008discontinuous}. As this property will be important in subsequent derivations, we recall the definition here.
\begin{definition} \label{def:upper-semicont}
A set-valued function $G:\R^m\rightrightarrows\R^m$ is said to be \emph{upper semicontinuous} at $x$ if for any $\e>0$ there exists a $\delta > 0$ such that for all $y\in B_\delta(x)$, $G(y)\subset B_{\e}(G(x))$.
\end{definition}

\section{Generalized Setup: Subspace-Constrained Optimization} \label{sec:general-setup}
The problem of minimizing \eqref{eqn:f-def} in a distributed setting may be viewed as the subspace-constrained optimization problem
\begin{equation} \label{eq:dist-opt-prob}
\min_{\substack{x_n\in \R^d,~ n=1,\ldots,N\\ x_1 = x_2 = \cdots = x_N }} \sum_{n=1}^N f_n(x_n).
\end{equation}
Rather than focus on the specific problem \eqref{eq:dist-opt-prob} we will consider optimization of general subspace-constrained optimization problems. This will significantly simplify notation by eliminating distributed-consensus specific notation and will improve the transparency of proofs. 


In Section \ref{sec:gen-framework} we will set up the general subspace constrained optimization problem to be considered in the rest of the paper and describe a generalization of \eqref{dynamics_CT} for addressing this problem. However, before considering the general problem it will be helpful to first derive some simple time changes. This will be done in Section \ref{sec:time-changes}. After a time change, the dynamics \eqref{dynamics_CT} admit an intuitive interpretation in terms of gradient descent with respect to a penalty function. This will become clear in Section \ref{sec:gen-framework}.

\subsection{Time Changes} \label{sec:time-changes}
The differential inclusion \eqref{dynamics_CT} may be expressed compactly as
\begin{equation} \label{eq:ODE-vec-form}
\dot \vx \in \beta_t(L\otimes I_{d})\vx - \alpha_t(\partial f_n(\vx_n))_{n=1}^N,
\end{equation}
where we let $\vx:\R\to\R^{Nd}$ be the vectorization $\vx := (\vx_1,\ldots,\vx_N)$, where $\vx_n:\R\to\R^d$ represents the state of agent $n$, and, as before, we assume $\alpha_t/\beta_t\to 0$. It will often be convenient to study this ODE under a time change. In particular, assuming $\alpha_t>0$ for $t\geq 0$, set $S(t) = \int_{0}^t \alpha_r \,dr$ and let $T(\tau)$ denote the inverse of $S(t)$ for $\tau\geq 0$, so that $T(S(t)) = t$. Letting $\vy(\tau) = \vx(T(\tau))$ we have
\begin{equation} \label{eq:ODE-alt-form1}
\dot \vy(\tau) \in \gamma_\tau(L\otimes I_{d})\vy(\tau) - (\partial f_n(\vy(\tau)))_{n=1}^N,
\end{equation}
where $\gamma_\tau = \frac{\beta_{T(\tau)}}{\alpha_{T(\tau)}} \to \infty$ as $\tau\to\infty$. Likewise, if we set $S(t) = \int_{0}^t \beta_r \,dr$ and let $T(\tau)$ denote the inverse of $S(t)$ we have
\begin{equation} \label{eq:ODE-alt-form2}
\dot \vy(\tau) \in (L\otimes I_{d})\vy(\tau) - \tilde \gamma_\tau(\partial f_n(\vy(\tau)))_{n=1}^N,
\end{equation}
where $\tilde \gamma_\tau = \frac{\alpha_{T(\tau)}}{\beta_{T(\tau)}} \to 0$ as $\tau\to\infty$. Thus, processes of the form \eqref{eq:ODE-alt-form1} or \eqref{eq:ODE-alt-form2}, with $\gamma_t \to \infty$ or $\tilde \gamma_t\to 0$ respectively, generalize dynamics of the form \eqref{dynamics_CT}.
When convenient we will study \eqref{eq:ODE-alt-form1} or \eqref{eq:ODE-alt-form2} (with associated parameter $\gamma_\tau$ or $\tilde \gamma_\tau$) in lieu of \eqref{dynamics_CT}. 

\subsection{Subspace-Constrained Optimization Framework} \label{sec:gen-framework}
Consider the optimization problem
\begin{align} \tag{P.1} \label{eq:opt-prob}
\min_{x\in R^M} & \quad h(x)\\
\text{subject to} & \quad  x^\T Q x = 0,
\end{align}
where $h:\R^M\to\R$ is a locally Lipschitz function and $Q\in\R^{M\times M}$ is a positive semidefinite matrix. For ease of notation we will denote the constraint set by
\begin{equation} \label{eq:constraint-def}
\C \coloneqq \{x\in \R^M:\,x^\T Qx = 0\}.
\end{equation}
Since $Q$ is positive semidefinite, $\C$ is precisely the set $\{x: Qx = 0\}$, i.e., the nullspace of $Q$; we write the constraint in its quadratic form because we will solve this problem using a penalization approach that connects directly with the quadratic form.
In the remainder of the paper we will focus on computing critical points in \eqref{eq:opt-prob}.

Consider the following dynamical system for solving \eqref{eq:opt-prob}:
\begin{equation}\label{eq:ODE1}
\dot \vx \in  -\partial h (\vx) -\gamma_t Q\vx,
\end{equation}
where the weight $\gamma_t\to\infty$. Note that solutions to \eqref{eq:ODE1} exist if $h$ is locally Lipschitz continuous (see Assumption \ref{a:h-loc-Lip} below).
Note that these may be viewed as the generalized gradient descent dynamics associated with the (time-varying) function $x\mapsto h(x) + \gamma_t \frac{1}{2}x^\T  Q x$, i.e., $$
\dot \vx \in -\partial_x \left( \gamma\frac{1}{2}\vx^\T  Q\vx + h (\vx) \right).
$$
The term $\gamma_t x^\T  Q x$ may be thought of as a quadratic penalty term that punishes deviations from $\C$ with increasing severity as $t\to\infty$.

\subsection{DGF as a Special Case} \label{sec:DGD-special-case}
The DGF dynamics \eqref{eq:ODE-alt-form2} for distributed optimization may be seen as a special case of this general framework in which we let the dimension be given by $M= Nd$, the state $x\in \R^{Nd}$ is given by the vectorization of all agents' states $x=\{x_n\}_{n=1}^N$, the objective function is given by $h(x) = \sum_{n=1}^Nf_n(x_n)$, and the penalty term is generated by setting $Q = (L\otimes I_d)$, where $I_d$ is the $d\times d$ identity matrix and $L$ denotes the graph Laplacian of $G$ given in Assumption \ref{a:G-connected-undirected}. 
In this setup, the constraint set $\calC$ is the consensus subspace, which is given by the nullspace of $(L\otimes I_d)$. If the $G$ is connected, this is the subspace of $R^{Nd}$ where $x_n=x_\ell$ for all $n,\ell$. 
It will be important later to note that under Assumption \ref{a:G-connected-undirected}, $Q$ has at least one zero eigenvalue (cf. Assumption \ref{a:Q-PSD1} below) and under Assumption \ref{a:step-size-CT} and the time change given in \eqref{eq:ODE-alt-form1} we have $\gamma_t\to\infty$. 

\section{Convergence to Critical Points} \label{sec:conv-to-CP}
In this section we show that \eqref{eq:ODE1} converges to critical points of $h$ restricted to $\calC$ (i.e., we prove Theorem \ref{thrm:cont-conv-cp}). 
Before proceeding, we will begin by introducing some conventions that will simplify notation.
Throughout this section, without loss of generality 
assume the coordinate system is rotated so that the constraint space is given by
\begin{equation} \label{eq:C-rotation}
\calC = \{x\in\R^M: x_{d+1}=\cdots =x_M = 0\},
\end{equation}
where we let\footnote{We note that we also used $d$ for the dimension of the domain of $f$ and $f_n$ in Section \ref{sec:intro}. Since, in the context of the distributed framework $\calC$ corresponds to the consensus subspace, which has dimension $d$, this does not result in a conflict of notation.} 
$$
d= \dim\C.
$$
Given a vector $x\in \R^M$, we will use the decomposition
$$x = (x_c,x_{nc}),$$
$x_c\in \R^d$, $x_{nc}\in\R^{M-d}$, where the subscripts indicate the ``constraint'' and ``not constraint'' components respectively.
In a slight abuse of notation, given $x_c\in \R^{d}$ we let
$$
h\vert_\C(x_c) := h(x_c,0).
$$
Given $x\in \R^M$ define
\begin{multline}
\partial_{x_c} h(x) := \{z\in \R^{d}: (z,y)\in \partial h(x),\\ 
\mbox{ for some } y\in \R^{M-d}\}.
\end{multline}
In a slight abuse of terminology, we say that $x^*=(x_c^*,x_{nc}^*)\in\R^M$ is a critical point of $h\vert_\C$ if $0\in \partial_{x_c} h(x_c^*,0)$, or equivalently, if $0\in \partial h\vert_\C(x_c^*)$.

We now present the assumptions we will use in the general framework. Because we are now studying the general subspace-constrained optimization framework, these assumptions pertain to \eqref{eq:opt-prob} and the optimization dynamics \eqref{eq:ODE1}, and are distinct from the previous assumptions made in the paper (which applied explicitly to the DGF framework). To distinguish these assumptions from those made earlier, previous assumptions have been numbered A.1., A.2., etc., while all subsquent assumptions will be numbered B.1., B.2., etc. 

\begin{newassumption} 
\label{a:h-loc-Lip}
  $h$ is locally Lipschitz continuous. 
\end{newassumption}
\begin{newassumption} 
\label{a:coercive-h}
 $h$ is coercive, i.e., $h(x)\to\infty$ as $\|x\|\to\infty$.
\end{newassumption}

\begin{newassumption} 
\label{a:Q-PSD1}
  $Q\in \R^{M\times M}$ is positive semidefinite with at least one zero eigenvalue.
\end{newassumption}
\begin{newassumption} \label{a:zero-measure-cp}
Let $\text{CP}_{h\vert_\C}$ be the critical points set of $h\vert_\C$. Assume that $\R\backslash h\vert_{\C}(\text{CP}_{h\vert_\C})$ is a dense set in $\R$.
\end{newassumption}
\begin{newassumption} \label{a:chain-rule-h}
For any absolutely continuous function $\vz:[0,\infty)\to \R^{d}$, 
$h\vert_\C$ satisfies the chain rule
$$
\ddt h\vert_\C(\vz(t)) = \langle v, \ddt \vz(t) \rangle, 
$$
for some $v\in \partial f(\vz(t))$, and almost all $t\geq 0$. 
\end{newassumption}
\begin{newassumption} \label{a:gamma_t}
$t\mapsto \gamma_t$ is bounded on compact intervals and satisfies
$\lim_{t\to\infty} \gamma_t = \infty$. 
\end{newassumption}

Assumption \ref{a:h-loc-Lip} ensures that \eqref{eq:ODE1} is well defined, Assumption \ref{a:coercive-h} ensures solutions to \eqref{eq:ODE1} remain in a compact set, and Assumption \ref{a:Q-PSD1} ensures that the constraint set $\calC$ is nonempty. Assumptions \ref{a:zero-measure-cp}--\ref{a:chain-rule-h} are technical assumptions required to ensure convergence to critical points. 

We will prove the following result that implies Theorem \ref{thrm:cont-conv-cp}.
\begin{theorem} \label{thrm:conv-to-cp-CT-h}
Let $\vx$ be a solution to \eqref{eq:ODE1} and suppose that Assumptions \ref{a:h-loc-Lip}--\ref{a:gamma_t} hold. Then,
\begin{itemize}
    \item [(i)] $\vx(t) \to \C$ as $t\to\infty$.
    \item [(ii)] $\vx(t)$ converges to the set of critical points of $h\vert_{\C}$ as $t\to\infty$.
\end{itemize}
\end{theorem}
We remark that under Assumptions \ref{a:h-loc-Lip}--\ref{a:coercive-h}, the set of critical points of $h|_\C$ is nonempty.
The proof of Theorem \ref{thrm:conv-to-cp-CT-h} will be given in Section \ref{sec:CT-conv-to-cp} below, and will rely on techniques from the theory of stochastic approximation and perturbed differential inclusions \cite{benaim2005stochastic,davis2020stochastic}. Before proceeding to the proof, we will first briefly review some relevant tools in the next section.

\subsection{Intermediate Results}
In order to prove Theorem \ref{thrm:conv-to-cp-CT-h}, we will use the following standard results from functional analysis. 

Before stating the first result, we recall that a function $v:[0,T]\to\R^m$ is said to belong to $L^2([0,T];R^m)$, or $L^2$ for short, if $\int_0^T \left([v(t)]_i\right)^2\dt < \infty$ for each $i=1,\ldots,m$, where $[v(t)]_i$ indicates extracting the $i$-th coordinate map of $v$. A sequence of functions $v_j\in L^2([0,T]; \R^m)$, $j=1,2,\ldots$ is said to be bounded in $L^2$ if 
$$
\sup_{j\geq 1}\int_{0}^T \left([v_j(t)]_i\right)^2\dt < \infty,
$$ 
for each $i=1,\ldots,m$.
\begin{lemma} \label{lemma:alaoglu}
Let $T>0$ and suppose that $\{v_j\}_{j\geq 1}$ is a sequence of functions $v_j$, bounded in $L^2([0,T],\R^m)$. Then there exists a subsequence $\{v_{j_\ell}\}_{\ell\geq 1}$ that converges weakly to some function $\hat v$ in $L^2([0,T],\R^m)$.
\end{lemma}
The above lemma is an immediate consequence of the well-known Banach-Alaoglu Theorem \cite{conway2019course}. Since we will only use the notion of weak convergence in this paper to apply Lemma \ref{label:mazur-theorem} after invoking Lemma \ref{lemma:alaoglu} (see proof of Lemma \ref{lemma:conv-to-z-limit}), we will not formally review the definition of weak convergence here, but refer readers to \cite{conway2019course}.
The next result, commonly known as Mazur's theorem (or Mazur's lemma) allows us obtain strongly convergent sequence from a weakly convergent one \cite{conway2019course}. 
\begin{theorem}[Mazur's Theorem] \label{label:mazur-theorem}
Suppose that $\{v_j\}_{j\geq 1}$ is a sequence in $L^2([0,T]; \R^m)$ that converges weakly to some $\hat v$ in $L^2([0,T]; \R^m)$. 
Then for each $j$ there exist a positive integer $n_j\geq j$ and numbers $\alpha_{i,j} \in [0,1]$, $i=j,\ldots,n_j$ satisfying $\sum_{i=j}^{n_j} \alpha_{i,j} = 1$ such that the sequence $\{\hat v_j\}_{j\geq 1}$ defined by the convex combination
$$
\hat v_j = \sum_{i=j}^{n_j} \alpha_{i,j} \hat v_i
$$
converges to $\hat v$ in $L^2([0,T]; \R^m)$ as $j\to\infty$, i.e., $\int_0^T \|v_j(t) - \hat v(t)\|^2\dt\to 0$ as $j\to\infty$.
\end{theorem}

\subsection{Convergence to Critical Points: Analysis} \label{sec:CT-conv-to-cp}
We now prove Theorem \ref{thrm:conv-to-cp-CT-h}. 
We begin with the following lemma that shows convergence to the constraint set. 

\begin{lemma}[Convergence to Constraint Set] \label{lemma:cont-consensus}
Let $\vx$ be a solution to \eqref{eq:ODE1} and suppose that Assumptions \ref{a:h-loc-Lip}--\ref{a:Q-PSD1} and \ref{a:gamma_t} hold. Then $\vx(t)\to \calC$.
\end{lemma}
We note that Assumption \ref{a:zero-measure-cp} is not needed for this result---it is only required to obtain convergence to critical points. 
In the proof of Lemma \ref{lemma:cont-consensus}, we will use the following conventions. 
Consistent with \eqref{eq:C-rotation} and Assumption \ref{a:Q-PSD1}, assume $Q$ is block diagonal with form
\begin{equation} \label{eq:Q-diag-form}
Q = \begin{pmatrix}
0 & 0\\
0 & \widehat Q
\end{pmatrix}
\end{equation}
where $\widehat Q\in\R^{(M-d)\times (M-d)}$ is positive definite and here $0$ denotes a zero matrix of appropriate dimension.
Let $\vx(t)$ be decomposed as
\begin{equation} \label{eq-x-ct-decomp}
\vx(t) =
\begin{pmatrix}
\vx_c(t)\\
\vxperp(t)
\end{pmatrix},
\end{equation}
where $\vx_c(t)\in\R^d$ and $\vxperp(t) \in \R^{M-d}$.

We now prove Lemma \ref{lemma:cont-consensus}.
\begin{proof}
By Assumption \ref{a:coercive-h} there exists some bounded set $K\subset\R^{M}$ such that, regardless of initialization, solutions to \eqref{eq:ODE1} reach $K$ and remain in $K$ thereafter. 
Thus, without loss of generality we may consider solutions to \eqref{eq:ODE1} initialized in $K$. 

Let $\e>0$, let $\overline{M}=\sup \{ \|v\|: v\in \partial h(x),~x\in K\}$, and let $\lambda_{\textup{min}}>0$ be the smallest eigenvalue of $\widehat Q$. Since $\gamma_t\to\infty$ we may choose some $T$ such that $\gamma_t \geq \frac{\frac{1}{\e} +\overline{M}}{\lambda_{\text{min}}\e}$ for all $t\geq T$. Using \eqref{eq:ODE1} we see that when $\|\vx_{nc}\|\geq \e$ we have $\ddt \|\vx_{nc}\|^2 \geq 1$. Thus, $\|\vx_{nc}(t)\| \leq \e$ after some finite time. Sending $\e\to 0$ completes the proof. 
\end{proof}

The remainder of this section will focus on proving convergence to critical points of $h\vert_\C$. Informally, we will prove the result by using use $h\vert_\C$ as a type of Lyapunov function. (More precisely, $h\vert_\C$ acts asymptotically as a Lyapunov function) We proceed as follows. First, we will define several important concepts that will be required in the proofs. Next,
Lemma \ref{lemma:conv-to-z-limit} will show that as $t\to\infty$, $\vx_c(t)$ asymptotically resembles the solution of a gradient-descent differential inclusion for $h\vert_\C$. 
Lemma \ref{lemma:conv-of-potential} will show that the Lyapunov function values $h\vert_\C(\vx_c(t))$ have a limit as $t\to\infty$. Finally, Lemma \ref{lemma:CT-conv-to-cp} will show convergence to critical points. Theorem \ref{thrm:conv-to-cp-CT-h} follows immediately from Lemmas \ref{lemma:cont-consensus} and \ref{lemma:CT-conv-to-cp}.

We now give several important definitions required through the remainder of the section.
Given $\tau\geq 0$, and a solution $\vx$ of \eqref{eq:ODE1}, define $\vx^\tau:[0,\infty)\to \R^M$ to be the \emph{shifted} solution curve
$$
\vx^\tau(t) := \vx(\tau + t).
$$
Note that $\vx^{\tau}$ captures the ``tail'' of $\vx$ after time $\tau$.

\begin{lemma} \label{lemma:conv-to-z-limit}
Suppose that Assumptions \ref{a:h-loc-Lip}--\ref{a:Q-PSD1} and \ref{a:chain-rule-h}--\ref{a:gamma_t} hold, and let $\vx(t) = (\vx_c(t),\vx_{nc}(t))$ be a solution to \eqref{eq:ODE1}.
Let $\{\tau_j\}_{j\geq 1}$ be a real-valued sequence of times satisfying $\tau_j\to\infty$. Given any $T>0$, there exists a subsequence of $\{\vx_c^{\tau_j}\}_{j\geq 1}$ that converges uniformly on the interval $[0,T]$ to some function $\vz:[0,\infty)\to \R^{d}$ satisfying
$$
\dot \vz(t) \in \partial h\vert_\C(\vz(t))
$$
for almost every $t\in[0,T]$.
\end{lemma}
\begin{proof}
The proof of this result is similar to the proof of Theorem 4.2 in \cite{benaim2005stochastic}. 

Let $\vx(t)$ be decomposed as in \eqref{eq-x-ct-decomp}. 
Let $T>0$ and consider the family of functions obtained by shifting $\vx_c$ by $\tau_j$ and restricting to  the interval $[0,T]$, i.e., the set of functions $\vx_c^{\tau_j}:[0,T]\to\R^d$, $j\geq 1$.  By \eqref{eq:ODE1} and \eqref{eq:Q-diag-form} we have
\begin{equation} \label{eq:x_c-stuff}
\vx_c^{\tau_j}(t) = \vx_c^{\tau_j}(0) + \int_0^t v_{j}(s)\ds,  
\end{equation}
where $v_{j}(s) \in \partial_{x_c} h(\vx^{\tau_j}(s))$. 
By Assumption \ref{a:coercive-h}, $\vx(t)$ remains in some compact set $K$. By Assumption  \ref{a:h-loc-Lip} there exists an $L>0$ such that $\frac{\|h(x) - h(y)\|}{\|x-y\|}\leq L$ for all $x,y\in K$. By Definition \ref{def:gen-grad} we have $\|v\|\leq L$ for all $v\in \partial h(x)$ and all $x\in K$. 
Thus $\|v_{j}(s)\|$ above is uniformly bounded for all $j$ and $s$, and hence $\{\vx_c^{\tau_j}\}_{j\geq 1}$ is an equicontinuous family of functions. By the Arzela-Ascoli theorem \cite{rudin1964principles}, there exits a subsequence of $\{\vx_c^{\tau_j}\}_{j\geq 1}$ converging uniformly to some function $\vz:[0,T]\to\R^d$. Without loss of generality, we will assume henceforth that the entire sequence $\{\vx_c^{\tau_j}\}_{j\geq 1}$ is identical to this subsequence so that 
$$
\vx_c^{\tau_j}(t) \to \vz(t),
$$
uniformly for $t\in[0,T]$ as $j\to\infty$.

Recalling \eqref{eq:x_c-stuff}, for $t\in [0,T]$ we obtain
\begin{align} \label{eq:z-chain}
    \vz(t) = \lim_{j\to\infty} \vx_c^{\tau_j}(t)  = \vz(0) + \lim_{j\to\infty} \int_0^t v_j(\tau)\dtau.
\end{align}
Note that, restricted to the interval $[0,T]$, $\{v_j\}_{j\geq 1}$ is a bounded sequence in $L^2$. By Lemma 
\ref{lemma:alaoglu}, there is a subsequence of $\{v_{j_\ell}\}_{\ell\geq 1}$ with weak limit $\hat v$ in $L^2([0,T]; \R^d)$. Without loss of generality assume that $\{v_j\}_{j\geq 1}$ is identical to this subsequence.
By Theorem \ref{label:mazur-theorem}, we see that there exists a sequence $\hat v_j(t)$ converging strongly to $\hat v(t)$ where 
$$
\hat v_j(t) = \sum_{i=j}^{n_j} v_j(t). 
$$
But since $v_j(t)\in \partial_{x_c} h(\vx_c^{\tau_j}(t),\vx^{\tau_j}(t))$ and $(\vx_c^{\tau_j}(t), \vx^{\tau_j}(t))\to (\vz(t),0)$ as $j\to\infty$, by the fact that $\partial h$ is upper semicontinuous (see Definition \ref{def:upper-semicont}) and convex we see that the the limit $\hat v(t)$ belongs to $\partial h\vert_\C(\vz(t))$. 
Thus we have that
$$
\vz(t) = \vz(0) + \int_0^t \hat v(\tau)\dtau,
$$
where $v(\tau) \in \partial h\vert_\C(\vz(\tau))$. 
\end{proof}

\begin{lemma} \label{lemma:conv-of-potential} Suppose Assumptions \ref{a:h-loc-Lip}--\ref{a:gamma_t} hold, and $\vx(t) = (\vx_c(t),\vx_{nc}(t))$ is a solution to \eqref{eq:ODE1}. Then $h\vert_\C(\vx_c(t))$ converges to a limit as $t\to\infty$. 
\end{lemma}
The proof of this lemma follows similar ideas to Section 3.3 in \cite{davis2020stochastic} which treats the classical centralized case in discrete time. The proof here handles the nontraditional subspace-constrained optimization framework in continuous time.
\begin{proof}
Note that, by Assumptions \ref{a:h-loc-Lip}--\ref{a:coercive-h}, $h\vert_\C$ is bounded from below. 
Without loss of generality, assume $\lim\inf_{t\to\infty} h\vert_\C(\vx_c(t)) = 0$.
Let $\e>0$ be a noncritical value of $h\vert_\C$ (i.e., $0\not\in \partial h\vert_\C(z)$ for any $z$ such that $h\vert_\C(z) = \e$). By Assumption \ref{a:zero-measure-cp} we may choose such an $\e$ to be arbitrarily close to zero.

Given $r\geq0$ define the $r$-sublevel set
$$
L_r := \{y\in\R^d: h\vert_\C(y) \leq r\}. 
$$
Note that $L_\e$ and $L_{2\e}$ are well separated in the sense that 
\begin{equation} \label{eq:Lr-thick-set}
\inf\{\|w-v\|:w\in K\cap L_\e, v\notin L_{2\e}\}>0.
\end{equation}
(See \cite{davis2020stochastic}, proof of Claim 1 in the proof of Proposition 3.5.) 
By hypothesis, $\vx_c(t)$ enters and exits $L_{\e}$ and $L_{2\e}$ infinitely often. Define the time $t_1$ to be the first time $t$ where the following two conditions hold:
\begin{itemize}
\item [1.] $\vx_c(t)$ exits $L_\e$ at time $t_1$, i.e., $\vx(t_1)\in L_\e$ and
$$
\vx(t_1+ \tau) \not\in L_\e 
$$
for all $\tau> 0$ sufficiently small.
\item [2.] After time $t_1$, $\vx_c(t)$ exits $L_{2\e}$ before returning to $L_{\e}$.
\end{itemize}
In other words, $t_1$ is the last time $\vx_c(t)$ exits $L_\e$ before leaving $L_{2\e}$.
Having defined $t_1$, define $t_{\text{exit},1}$ to be the first time that $\vx_c(t)$ exits $L_{2\e}$ after $t_1$. For $j\geq 2$, iteratively define $t_j$ and $t_{\text{exit},j}$ in a similar manner and note that
$$
t_1 < t_{1,\textup{exit}} < t_2 < t_{2,\textup{exit}} < \cdots.
$$
This iterative procedure for constructing $t_j$ and $t_{j,\textup{exit}}$ terminates after a finite number of iterations for arbitrary $\e>0$ if and only if 
$h\vert_\C(\vx_c(t))$ converges to a limit. 

We will now show that the process must terminate after a finite number of iterations. For the sake of contradiction, suppose to the contrary that the assertion is false. 

Note that by Assumption \ref{a:coercive-h}, for each initialization, there is some compact set $K$ such that $\vx(t)\in K$ for all $t\geq 0$. By Assumption \ref{a:h-loc-Lip}, we have $\sup\{\|v\|:~v\in \partial_{x_c} h(x), ~x\in K\}< \infty$ (see proof of Lemma \ref{lemma:conv-to-z-limit}). Since $L_{2\e}\backslash L_{\e}$ is a ``thick'' set, i.e., \eqref{eq:Lr-thick-set} holds, it follows that there is some minimum time $t_{\textup{cross}}>0$ such that $t_{j,\textup{exit}} - t_j \geq t_{\textup{cross}}$ for all $j$. 
Thus our sequence of times $\{t_j\}_{j\geq 1}$ satisfying conditions 1--2 above, also satisfies $t_j\to\infty$.

Consider the sequence $\{\vx^{t_j}_c(t)\}_{j\geq 1}$. Let $T>0$. By the previous theorem, there exists a subsequence converging to $\vz(t)$ satisfying $\dot \vz(t) \in \partial h\vert_\C(\vz(t))$ for almost all $t\in [0,T]$. By construction, we have $h\vert_\C(\vz(0)) = \e$. 
Moreover, by our choice of $\e$, we have
$0\notin \partial h\vert_\C(\vz(0))$. By the upper semicontinuity of $\partial h\vert_\C$ (Definition \ref{def:upper-semicont}), we have that $0\not\in \partial h\vert_\C(\vz(0))$ for all $z$ in an open ball about $\vz(0)$. Using Assumption \ref{a:chain-rule-h}, it follows that 
\begin{equation} \label{eq:descent}
h\vert_\C(\vz(T)) < \sup_{t\in [0,T]} h\vert_\C(\vz(t))\leq h\vert_\C(\vz(0)) = \e.
\end{equation}
Let $\delta = \frac{1}{2}(h\vert_\C(\vz(T)) - h\vert_\C(\vz(0)))$. 
For $J>0$ sufficiently large, we may make $\sup_{t\in[0,T]} \|\vx^{t_j}(t) - \vz(t)\|$ arbitrarily small for all $j\geq J$. By continuity of $h$, we may thus make $\sup_{t\in [0,T]} \|h\vert_\C(\vx^{t_j}(t)) - h\vert_\C(\vz(t))\| < \delta$ for all $j\geq J$. 
Hence, $\vx^{t_j}(t)$ belongs to $L_\e$ at time $t_j+T$. But, by construction of $t_j$, $\vx^{t_j}(t)$ must exit $L_{2\e}$ before re-entering $L_\e$. This is impossible since, for all $j$ sufficiently large we have
\begin{align}
\sup_{t\in [0,T]} h\vert_\C(\vx_c^{t_j}(t)) \leq & \sup_{t\in [0,T]} h\vert_\C(\vz(t)) \\
& + \sup_{t\in [0,T]} \Big|h\vert_\C(\vx_c^{t_j}(t)) - h(\vz(t))\Big|\\
\leq & ~ 2\e,
\end{align}
where the second inequality follows from \eqref{eq:descent} and the uniform convergence of $\vx^{\tau_j}$ to $\vz$ on $[0,T]$.
Thus, the iterative procedure outlined above terminates after finite iterations and $\limsup_{t\to\infty} h\vert_\C(\vx_c(t)) = 0$. 

\end{proof}

\begin{lemma} \label{lemma:CT-conv-to-cp}
Suppose Assumptions \ref{a:h-loc-Lip}--\ref{a:gamma_t} hold and $\vx(t) = (\vx_c(t),\vx_{nc}(t))$ is a solution to \eqref{eq:ODE1}. Then $\vx_c(t)$ converges to the set of critical points of $h\vert_{\C}$.
\end{lemma}
\begin{proof}
Suppose that $x^*\in \R^d$ is a limit point of $\vx_c(t)$, but that $0\not\in \partial h\vert_\C(x^*)$. Choose times $\{\tau_j\}_{j\geq 1}$, $\tau_j\to\infty$, such that $\vx(\tau_j)\to x^*$. Let $T>0$. Then there exists some function $\vz(t)$ such that $\vx^{\tau_j}(t)\to \vz(t)$ uniformly for $t\in [0,T]$ and $\dot \vz(t) \in \partial h\vert_\C(\vz(t))$. Note that, by construction, $\vz(0) = x^*$. Recalling that $0\not\in \partial h\vert_\C(x^*)$, by the upper semicontinuity of $\partial h\vert_\C$, we have that $0\not\in \partial h\vert_\C(z)$ for all $z$ in an open ball about $x^*$. Using Assumption \ref{a:chain-rule-h}, it follows that
$$
h\vert_\C(\vz(t)) < h(x^*)
$$
for any $t>0$.
However, for $t\in (0,T]$ we see that 
\begin{align}
h\vert_\C(\vz(t)) = & \lim_{j\to\infty} h\vert_\C(\vx^{\tau_j}_c(t))\\
= & \lim_{s\to\infty} h\vert_\C(\vx(s)) = h\vert_\C(x^*),
\end{align}
where the last two equalities follow by Lemma \ref{lemma:conv-of-potential}. This contradicts our hypothesis. Hence, if $x^*$ is a limit point of $x_c(t)$, then $0\in \partial h\vert_\C(x^*)$.

\end{proof}

\begin{example} [Necessity of Coercivity]
The following example demonstrates that if we do not assume that $h$ is coercive, then Lemma \ref{lemma:cont-consensus} may fail to hold (i.e., $\vx(t)\not\to \C$, or equivalently, agents don't converge to consensus in DGF). 

Let $h:\R^2 \to \R$ and $Q\in \R^{2\times 2}$ be given by
$$
h(x) =
-\frac{1}{2}x^\T
\begin{pmatrix}
0 & 1\\
1 & 1
\end{pmatrix}
x \quad \text{ and } \quad Q =
\begin{pmatrix}
1 & 0\\
0 & 0
\end{pmatrix}
$$
so that $\calC = \{x\in\R^2: x_1 = 0 \}$.
Then the ODE \eqref{eq:ODE1} is given by
$$
\dot \vx =
\begin{pmatrix}
- \gamma_t & 1\\
1 & 1
\end{pmatrix}
\vx.
$$
Note that the first quadrant is an invariant set under these dynamics. Let $x_0 = (1,1)$ and suppose that $\gamma_t = t$. Then we have $\dot \vx_2(t) \geq \vx_2(t)$ and using a Gronwall type argument we have $\vx_2(t) \geq e^{t}$ for all $t\geq 0$. Thus we have
\begin{align}
\dot \vx_1(t) & = -t\vx_1(t) + \vx_2(t)\\
& \geq -t\vx_1(t) + e^{t}.
\end{align}
Suppose that $\vx_1(t)\to 0$ (i.e., $\vx\to\C$). Then for any $\e>0$ there exists a $T\geq 0$ such that $\vx_1(t) < \e$ for all $t\geq T$. Thus,
\begin{align}
\dot \vx_1(t) & \geq -t\e + e^{t}>0
\end{align}
for all $t\geq T$ sufficiently large, which contradicts the supposition that $\vx_1(t)\to 0$. On the other hand, it is worth noting that if we suppose $\gamma_t$ increases faster than $e^{t}$ then it can be shown that $\vx(t) \to\C$.
\end{example}

\section{Stable-manifold Theorem} \label{sec:CT-stable-manifold}
In this section we will establish the stable-manifold theorem for DGF. More precisely, we will prove a stable-manifold theorem for the general ODE \eqref{eq:ODE1} which will imply Theorem \ref{thrm:non-convergence-saddles}.

Let $x^*$ be a saddle point of interest. We will make the following assumptions.\footnote{Because the analysis in this section will take place locally under Assumption \ref{a:h-C2}, we will use the standard gradient $\nabla h$ rather than the generalized gradient.}
\begin{newassumption} \label{a:h-C2}
$h$ is of class $C^2$ in a neighborhood of $x^*$.
\end{newassumption}

We will also make the following technical assumptionn which states that the eigenvectors of $h$ are continuous near saddle points (see the discussion near Assumption \ref{a:eigvec-continuity} for situations where this assumption holds).
\begin{newassumption}\label{a:eigvec-continuity-h}
Let $x^*$ be a regular saddle point of $h$. Assume that the eigenvectors of $\Hess h(x)$ are continuous near $x^*$ in the sense that for each $x$ near $x^*$, there exists an orthonormal matrix $U(x)$ that diagonalizes $\Hess h(x)$ such that $x\mapsto U(x)$ is continuous at $x^*$.
\end{newassumption}
We reiterate that this assumption is relatively mild and should be satisfied by most functions encountered in practice. It is required to rule out certain pathological cases, as illustrated next. 

\begin{example} \label{example-eigval-problem}
Consider the scalar function on $\R^2$ given (in polar coordinates) by
\[
f(r,\theta) = e^{-1/r^2} \cos(\theta)
\]
We note that $f$ is $C^\infty$ but not analytic. One may compute (for $r$ small)
\[
\begin{pmatrix}
\partial_{rr} f & \partial_{r\theta} f\\ \partial_{\theta r} f & \partial_{\theta \theta} f
\end{pmatrix} \approx e^{-1/r^2}\begin{pmatrix} \frac{4\cos(\theta)}{r^6} & -\frac{2 \sin(\theta)}{r^3} \\ -\frac{2 \sin(\theta)}{r^3} & -\cos(\theta) \end{pmatrix}.
\]
If $\theta = 0$ one eigenvalue is $4e^{-1/r^2}r^{-6}$, with eigenvector in the $r$ direction (i.e. the $x$ direction), while the other eigenvalue is negative. On the other hand, for $\theta = \pi$ there is an eigenvalue in the $\theta$ direction with eigenvalue $e^{-1/r^2}$ in polar coordinates (which becomes $e^{-1/r^2} r^{-1}$ in terms of the $x$ coordinates), and the other eigenvalue is positive. These eigenvalues are continuous as they approach $r=0$, but the eigenvectors do not vary continuously.
\end{example}

\begin{remark}
In Assumption \ref{a:eigvec-continuity} we require that each individual $f_n$ has continuous eigenvectors. Per Section \ref{sec:DGD-special-case},
the function $h$, in the context of the general setting, corresponds to the function $(x_n)_{n=1}^N\mapsto \tilde f_n(x_n)$ in the context of DGF. Because each $f_n$ in the sum depends only on $x_n$, the Hessian of $\tilde f$ is block diagonal, and continuity of the eigenvectors of the individual $f_n$'s implies continuity of the eigenvectors of $\tilde f$.
\end{remark}

The following theorem demonstrates the existence of a stable manifold near regular saddle points.
\begin{theorem} \label{thrm:main-continuous}
Suppose that  $x^*$ is a regular saddle point of $h\vert_{\C}$ and Assumptions \ref{a:Q-PSD1} and 
\ref{a:gamma_t}--\ref{a:eigvec-continuity-h} hold.
Assume the weight function $t\mapsto \gamma_t$ is $C^1$. Let $q$ denote the number of negative eigenvalues of $\Hess h\vert_{\C}(x^*)$. Then there exists a $C^1$ manifold $\calS\subset [0,\infty)\times\R^{M}$ with dimension $M-q+1$ such that the following holds: For all $t_0$ sufficiently large, a solution $\vx$ to \eqref{eq:ODE1} converges to $x^*$ if and only if $\vx$ is initialized on $\calS$, i.e., $\vx(t_0) = x_0$, with $(t_0,x_0)\in\calS$.
\end{theorem}

\begin{remark}[Regarding Initialization and Theorem \ref{thrm:non-convergence-saddles}] 
In the above theorem, the stable manifold is constructed from the set of time-state pairs $(t_0,x_0)$ that yield convergence to the saddle point. Thus, it is a subset of $\R\times \R^M$. Constructing the stable manifold this way is particularly useful when studying discrete-time algorithms \cite{SMKP2020saddles}. In particular, $\calS$ as constructed above is a Lyapunov unstable set, which will allow us to show that discrete-time stochastic processes are repelled from it.
However, from a practical perspective, one generally has a fixed initial time $t_0$ and then chooses a corresponding initial state $x_0$. 
Given a fixed initial time $t_0$, the time-slice of the stable manifold, given by 
$$\calS_{t_0} :=\{x_0\in \R^M: (t_0,x_0)\in \calS\}$$ represents the set of initial states under which \eqref{eq:ODE1} converges to $x^*$ starting at time $t_0$. $\calS_{t_0}$ is a smooth $(M-q)$-dimensional manifold living in the state space $\R^M$. Thus, if $q>1$, then $\calS_{t_0}$ has Lebesgue measure zero in $\R^M$. 
\end{remark}

\subsection{Proof of Theorem \ref{thrm:main-continuous}} \label{sec:stable-manifold-C1}
We will break the proof of Theorem \ref{thrm:main-continuous} into two main parts. Lemma \ref{lemma:stable-man-exist} demonstrates existence of the stable manifold, but does not show smoothness. Lemma \ref{lemma:manifold-smoothness} shows that the manifold is smooth. Lemmas \ref{lemma:stable-man-exist}--\ref{lemma:manifold-smoothness} together establish Theorem \ref{thrm:main-continuous}.

We begin with the following preliminary lemma.
\begin{lemma} \label{lemma:g-existence}
Suppose Assumptions \ref{a:Q-PSD1} and \ref{a:h-C2}--\ref{a:eigvec-continuity-h} hold and suppose that $0$ is a regular saddle point of $h\vert_\calC(0)$.
There exists a function $g:[0,\infty)\to \R^M$ such that (i)
$\nabla h(g(\gamma)) + g(\gamma)^\T Q = 0$ for all $\gamma$ sufficiently large and (ii) $g(\gamma)\to 0$ as $\gamma\to\infty$. Moreover, the arc length of $\{g(\gamma): \gamma\geq \gamma_0\}$ is finite, where $\gamma_0$ is a sufficiently large constant, i.e.,
\begin{equation} \label{eq:g-bdd-arc-length}
\int_{\gamma_0}^\infty |g'(s)|\ds < \infty.
\end{equation}
\end{lemma}
In words, the idea of the lemma is the following: We are interested in 0 as a critical point of $h\vert_\calC$. Of course, 0 may not be a critical point of the  \emph{penalized} function $h(x)+\gamma \frac{1}{2} x^\T Qx$, for a fixed $\gamma\geq 0$. But for sufficiently large penalty (i.e., for $\gamma$ sufficiently large), there is a critical point of the penalized function near 0. The location of this critical point is given by $g(\gamma)$. As we take $\gamma\to\infty$, the critical point of the penalized function converges to 0. 
In analyzing the dynamics \eqref{eq:ODE1}  it will typically be convenient to recenter about the point $g(\gamma_t)$ at any given time $t$.  
The proof of the lemma is given below.
\begin{proof}
The lemma will follow by repeated application of the implicit function theorem. Without loss of generality, assume that the constraint set is given by $\calC = \myspan\{e_1,\ldots,e_d\}$, i.e., the span of the first $d$ canonical vectors.
Let $x\in\R^M$ be decomposed as $x=(x_c,x_{nc})$, where $x_c\in\R^d$ refers to the `constraint' component and $x_{nc}\in\R^{M-d}$ refers to the `not constraint' component of $x$. Let $G_c:\R^M\to\R^d$ be given by
\begin{align}
G_c(x_c,x_{nc}) & :=  D_{x_c} \left( h(x_c,x_{nc}) + x^\T  Q x \right)\\
& = D_{x_c} h(x_c,x_{nc}),
\end{align}
where the second line follows from the fact that, by construction, $Q$ is null in directions along the constraint set. Observe that $G_c$ is $C^1$ and $G_c(0,0) = 0$. Recalling that $\Hess h\vert_\calC(0)$ is invertible (i.e., $D_{x_c}^2 h(x_c,x_{nc})\vert_{(x_c,x_{nc})=(0,0)}$ is invertible), the implicit function theorem implies that there exists a unique, $C^1$ function $x^{c}:\R^{M-d}\to\R^d$ such that $$G_c(x^c(x_{nc}),x_{nc})=0$$
for $x_{nc}$ in a neighborhood of zero.

Given that $\calC = \myspan\{e_1,\ldots,e_d\}$, the matrix $Q$ takes the form  $Q = \begin{pmatrix}0 & 0 \\ 0 & Q_{nc} \end{pmatrix}$, where $0\in\R^{d\times d}$ is the zero matrix, and $Q_{nc}\in\R^{(M-d)\times (M-d)}$ is positive definite.

For $\tau\geq 0$, let $G_{nc}:\R^M\to\R^{M-m}$ be given by
$$G_{nc}(\tau,x_{nc}) := \tau D_{x_{nc}} h(x^c(x_{nc}),x_{nc}) + x_{nc}^\T Q_{nc},$$
where, in an abuse of notation, by $D_{x_{nc}} h(x^c(x_{nc}),x_{nc})$ we mean
$D_{x_{nc}}h$ evaluated at $(x^c(x_{nc}),x_{nc})$.
Note that $G_{nc}$ is $C^1$, $G_{nc}(0,0) = 0$, and $D_{x_{nc}} G_{nc}(\tau,x_{nc})\vert_{(\tau,x_{nc})=(0,0)} = Q_{nc}$, which is invertible. By the implicit function theorem there exists a function $x^{nc}(\tau)$ such that $G_{nc}(\tau,x^{nc}(\tau))=0$ for $\tau$ near zero.

For $\gamma> 0$ sufficiently large let $g(\gamma) := (x^c(x^{nc}(1/\gamma)),x^{nc}(1/\gamma))$. By construction, for all $\gamma$ sufficiently large, $\frac{1}{\gamma}\nabla h(x) + x^\T Q = 0$, or equivalently, $\nabla h(x) + \gamma x^\T Q = 0$ for $x=g(\gamma)$.\footnote{As $h$ is scalar valued, the notation $\nabla$ and $D$ are both used refer to the gradient and are used interchangeably.}

The integrability claim \eqref{eq:g-bdd-arc-length} follows by noting that $\tau\mapsto \hat x(\tau):= (x^{c}(x^{nc}(\tau)),x^{nc}(\tau))$ is $C^1$ (by our use of the implicit function theorem), and after a change of variables the integral \eqref{eq:g-bdd-arc-length} is equivalent to $\int_0^{\tau_1} |D_\tau \hat x(\tau)|\dtau$ for some finite $\tau_1$. Since $\hat x$ is $C^1$, the integral is finite.
\end{proof}

The next lemma establishes the existence of a stable manifold. The proof technique relies on an adaptation of the classic Perron-Lyapunov method (see e.g. Chapter 4 in \cite{chicone2006ordinary}) tailored to the particular nonautonomous dynamical system \eqref{eq:ODE1}.
\begin{lemma} \label{lemma:stable-man-exist}
Suppose Assumptions \ref{a:Q-PSD1} and \ref{a:gamma_t}--\ref{a:eigvec-continuity-h} hold and let $h$, $\gamma_t$, $p$ and $x^*$ be as in Theorem \ref{thrm:main-continuous}. Then there exists a manifold $\calS\subset \R\times  \R^{M}$ with dimension $M-p+1$ such that the following holds: For all $t_0$ sufficiently large, a solution $\vx$ to \eqref{eq:ODE1} converges to $x^*$ if and only if $\vx$ is initialized on $\calS$, i.e., $\vx(t_0) = x_0$ with $(t_0,x_0)\in\calS$.
\end{lemma}
\begin{proof}
\noindent 1. (\textit{Recenter}) Without loss of generality we will  assume that $x^* = 0$. By
Lemma \ref{lemma:g-existence} there exists a function $g\in C^1([0,\infty); \R^M)$ such that, for each $\gamma \geq 0$ sufficiently large, $g(\gamma)$ is a critical point of the penalized function $h(x)+\gamma x^\T Qx$ and $g(\gamma)\to 0$ as $\gamma\to\infty$.

Letting $\vy(t) = \vx(t) - g(\gamma_t)$ we see that $\vx$ is a solution to \eqref{eq:ODE1} if and only if $\vy$ is a solution to
\begin{equation}\label{eq:ODE2}
\dot \vy = -\nabla_x h(\vy + g(\gamma_t)) - \gamma_t Q(\vy + g(\gamma_t)) - g'(\gamma_t)\dot \gamma_t,
\end{equation}
where we use the notation $g'(\gamma)$ to denote $D g(\gamma)$.
For $t\geq 0$ let
\begin{equation} \label{eq:A-matrix-def}
A(t) \coloneqq -\Hess_x \left( h(x) + \gamma_t x^\T  Q x \right)\big\vert_{x = g(\gamma_t)}
\end{equation}
and let
\begin{equation} \label{eq:F-def}
F(y,t) \coloneqq -\nabla_x h(y + g(\gamma_t)) - \gamma_t Q(y + g(\gamma_t)) - A(t)y
\end{equation}
so that we may express \eqref{eq:ODE2} as
\begin{equation}\label{eq:ode-y}
\dot \vy(t) = A(t)\vy(t) + F(\vy(t),t) - g'(\gamma_t)\dot \gamma_t.
\end{equation}

2. (\textit{Diagonalize}) For each $t\geq 0$, let $U(t)$ be a unitary matrix that diagonalizes $A(t)$ (which is possible as $A(t)$ is always symmetric), so that
\begin{equation} \label{def:Lambda-t}
\Lambda(t) \coloneqq U(t)A(t)U(t)^\T ,
\end{equation}
where $\Lambda(t)$ is diagonal. Since $\gamma_t\in C^1$, 
by Assumption \ref{a:eigvec-continuity-h} we may construct $U(t)$ as a differentiable function with $U(t)$ that converges to some fixed matrix
as $t\to\infty$ (or, equivalently, as $g(\gamma_t)\to 0$).  Changing coordinates again, let $\vz(t) = U(t)\vy(t)$ so that $\vy$ is a solution to \eqref{eq:ode-y} if and only if $\vz$ is a solution to
\begin{align}
\dot \vz(t)  = & U(t)\dot \vy(t) + \dot U(t) \vy(t)\\
= & U(t)\Big( A(t)U(t)^\T \vz(t) + F(U(t)^\T \vz(t),t)\\
& \hspace{5em} - g'(\gamma_t)\dot \gamma_t \Big) + \dot{U}(t) U(t)^\T  \vz(t)
\end{align}
Letting
\begin{equation} \label{eq:def-F-tilde}
\tilde F(z,t) \coloneqq U(t)F(U(t)^\T  z,t) + \dot U(t) U(t)z,
\end{equation}
the above is equivalent to
\begin{equation}\label{eq:ODE-z}
\dot \vz(t) = \Lambda(t)\vz(t) + \tilde F(\vz(t),t) - U(t)g'(\gamma_t)\dot \gamma_t.
\end{equation}
Note that $F(0,t) = 0$ and $F(y,t) = o(|y|^2)$ for $t\geq 0$. Consequently, for any $\epsilon>0$ there exists an $r>0$ and a $T\geq 0$ such that for all $t\geq T$ we have
\begin{equation} \label{eq:Lip-F}
|\tilde F(z,t) - \tilde F(\tilde z,t)| \leq \e |z-\tilde z|, \quad\quad \forall ~z,\tilde z\in B_r(0).
\end{equation}

\bigskip
\noindent 3. (\textit{Compute Stable Solutions}) Let $\lambda_1(t),\ldots,\lambda_M(t)$ denote the eigenvalues of $\Lambda(t)$. Without loss of generality, we may assume that the eigenvalues are ordered so each $\lambda_i(t)$ varies smoothly in $t$ (see Theorem II.5.1 in \cite{katoBook}.) Let 
\begin{equation}\label{def:B}
B := -\Hess h\vert_\calC(0).
\end{equation}
and let $\lambda_1,\ldots,\lambda_d$ denote the eigenvalues of $B$.
By Lemma \ref{lemma:eigvalue-convergence} in the appendix, for each eigenvalue $\lambda_i$ of $B$, there exists an eigenvalue $\lambda_i(t)$ of $\Lambda(t)$ such that $\lambda_i(t)\to\lambda_i$. Moreover, for each remaining eigenvalue of $\Lambda(t)$ there holds $\lambda_i(t)\to-\infty$.
Given the limits established for each $\lambda_i(t)$, there exists a time $T$ sufficiently large such that for each $i$ the sign of $\lambda_i(t)$ remains constant for $t\geq T$.

Without loss of generality assume that the coordinates are ordered so that the first $\ns <M$ diagonal entries of $\Lambda(t)$ are negative
and the remaining $M-\ns$ diagonal entries are positive for all $t$ sufficiently large. (The notation $n_s$ is indicative of number of ``stable'' eigenvalues.) Let $\Lambda(t)$ be decomposed as
\begin{equation} \label{eq:Lambda-decomposition}
\Lambda(t) =
\begin{pmatrix}
  \Lambda^{s}(t) & 0 \\
  0 & \Lambda^u(t)
\end{pmatrix}
\end{equation}
where $\Lambda^{s}(t)\in \R^{\ns\times \ns}$ and $\Lambda^u(t) \in \R^{(M-\ns)\times (M-\ns)}$ denote the `stable' and `unstable' diagonal submatrices respectively.
Let
\begin{align} \label{eq:V-def}
V^s(t_2,t_1) & \coloneqq
\begin{pmatrix}
e^{\int_{t_1}^{t_2} \Lambda^s(\tau)\,d\tau} & 0\\
0 & 0\\
\end{pmatrix},\\
~\\
V^u(t_2,t_1) & \coloneqq
\begin{pmatrix}
0 & 0\\
0 & e^{\int_{t_1}^{t_2} \Lambda^u(\tau)\,d\tau}\\
\end{pmatrix}.
\end{align}
By construction we have $\limsup_{t\to\infty} \lambda_j(t) < 0$,  $j=1,\ldots,n_s$. Hence, we may choose an $\nu > 0$ such that $\lambda_j(t) < -\nu < 0$ for $j=1,\ldots, n_s$ and all $t$ sufficiently large.
We may also choose constants $\sigma> 0$ and $K>0$ such that the following estimates hold
\begin{align}
\label{eq:ex-estimate-s}
\|V^s(t_2,t_1)\| & \leq Ke^{-(\nu+\sigma)(t_2-t_1)}, \quad \quad t_2\geq t_1\\
\label{eq:ex-estimate-u}
\|V^u(t_2,t_1)\| & \leq Ke^{\sigma (t_2-t_1)}, \quad\quad\quad~~~ t_2\leq t_1.
\end{align}
Let $t_0\in \R$, $a^s\in \R^{n_s}$, and $t\geq t_0$, and consider the integral equation
{\small
\begin{multline} \label{eq:integral-eq0}
\vu(t,(t_0,a^s)) = V^s(t,t_0)
\begin{pmatrix}
a^s\\
0
\end{pmatrix}\\
+ \int_{t_0}^t V^s(t,\tau)\bigg(\tilde F(\vu(\tau,(t_0,a^s)),\tau) -U(\tau)g'(\gamma_\tau)\dot\gamma_\tau \bigg)\,d\tau\\
 - \int_{t}^\infty V^u(t,\tau)\bigg( \tilde F(\vu(\tau,(t_0,a^s)),\tau) -U(\tau)g'(\gamma_\tau)\dot\gamma_\tau  \bigg)\,d\tau,
\end{multline}
}where $\vu:\R\times \R\times \R^{n_s}\to\R^M$. To be precise, we have included the initial time $t_0$ as a parameter in $\vu$. However, in most of our analysis $t_0$ will be fixed. Thus, in an abuse of notation, we will generally suppress the $t_0$ argument and only specify $\vu$ in terms of the arguments $t$ and $a^s$, i.e., $\vu(t,a^s)$. 

Suppose $\e < \frac{\sigma}{4K}$ and let $r$ and $T$ be chosen so that \eqref{eq:Lip-F} holds for all $t\geq t_0\geq T$.
Using a standard contraction mapping argument for  successive approximations (see, e.g., \cite{Perko_ODE}), it is straightforward to verify that \eqref{eq:integral-eq0} has a unique (continuous in $t$) solution for all $a^s$ sufficiently small and $t_0$ sufficiently large, and that the solution satisfies
\begin{equation} \label{eq:u-bound}
|\vu(t,a^s)| \leq 2K(1+|a^s|)e^{-\nu(t-t_0)}.
\end{equation}

If $t\mapsto \vu(t,a^s)$ is continuous and solves \eqref{eq:integral-eq0} then, $\vu(t,a^s)$ is differentiable in $t$ and solves \eqref{eq:ODE-z} with componentwise initialization $\vu_i(t_0,a^s) = a^s_i$ for $i=1,\ldots,n_s$. This follows by differentiating the right hand side of \eqref{eq:integral-eq0} in $t$.

\bigskip
\noindent 4. (\textit{Construct Stable Manifold})
We now construct the stable set $\calS$ corresponding to the ODE \eqref{eq:ODE-z}.
For each $z_0^s\in B_{\frac{r}{3}}(0)\subset \R^{n_s}$ let $\vu(\cdot,z_0^s)$ be the (unique) solution to \eqref{eq:integral-eq0}. For each $t\in[T,\infty)$ define the component map $\psi_j:\R\times \R^{n_s}\to \R$ by
\begin{equation}\label{eq:psi-S-def}
\psi_j(t_0,z_0^s) \coloneqq \vu_j(t_0,(t_0,z_0^s)), \quad j=n_s+1,\ldots,M,
\end{equation}
and let $\psi = (\psi_j)_{j=n_s+1}^M$.
In words, $\psi:\R^{n_s}\to \R^{M-n_s}$ takes as input an initial time $t_0$ and ``stable'' coordinates $z_0^s\in \R^{n_s}$ and returns the corresponding ``unstable'' coordinates so that the point $(z^s_0,\psi(t_0,z^s_0))\in \R^M$ is a stable initialization of the ODE at time $t_0$, that is, $z_0^u=\psi(t_0,z^s_0)$ is the unique point in $\R^{M-n_s}$ such that if $\vx(t_0) = (z^s_0, z_0^u)$ then $\vx(t)\to 0$ as $t\to\infty$.

The stable manifold (with respect to \eqref{eq:ODE-z}) is given by
$$
\calS \coloneqq\{(t_0,z_0^s,\psi(t_0,z_0^s)), t_0\geq T, \, z_0^s\in \R^k\cap B_{\frac{r}{3}}(0)\}.
$$

By Lemma \ref{lemma:stable-in-S} we see that $\calS$ contains all stable initializations $(t_0,z_0)$. That is, if $\vz$ is a solution to \eqref{eq:ODE-z} with $\vz(t_0) = z_0$ and $\vz(t)\to 0$, then $(t_0,z_0)\in \calS$.

Having constructed $\calS$ (the stable manifold for \eqref{eq:ODE-z}) the stable manifold  for \eqref{eq:ODE1}, denoted here by $\tilde \calS$, is obtained by an appropriate change of coordinates, $\tilde \calS \coloneqq\{(t,x)\in \R\times\R^M:~ U(t)(x-g(\gamma_t)) \in \calS \}.$
\end{proof}

Finally, the fact that $\calS$ is a $C^1$ manifold will be shown in the following lemma.
\begin{lemma} \label{lemma:manifold-smoothness}
Assume the hypotheses of Theorem \ref{thrm:main-continuous} hold. Then the stable manifold $\calS$ is of class $C^1$. That is, the maps $\psi_j$, $j=n_s+1,\ldots,M$, defined in \eqref{eq:psi-S-def} are $C^1$. Moreover, $\frac{\partial \psi_j(t_0,0)}{\partial a_i^s} = 0$, $j=n_s+1,\ldots,M$, $i=1,\ldots,n_s$ and all $t_0$.
\end{lemma}
We remark that the significance of the statement that $\frac{\partial \psi_j(t_0,0)}{\partial a_i^s} = 0$ in the lemma above is that it establishes that the stable eigenspace of \eqref{eq:ODE-z} is tangential to the stable manifold at 0. This is analogous to standard properties of the classical stable manifold \cite{coddington1955theory}. However, we note that when constructing $\psi$ in this case, we have recentered about $g(\gamma_t)$ and rotated by the time-varying $U(t)$. Thus, the stable eigenspace relative to which the manifold is tangential is a  time-varying object.
We now prove the lemma. 
\begin{proof}
Let $\vu(t,a^s)$ be the solution to \eqref{eq:integral-eq0} with stable initialization $a^s$ at time $t$. We will begin by establishing the existence of derivatives of $\vu$ with respect to the coordinates of $a^s$. Recalling \eqref{eq:psi-S-def}, this is equivalent to studying the partial derivatives of $\psi_j$. 

Fix a coordinate $i\in\{1,\ldots,\ns\}$. We will compute the vector of partial derivatives $(\frac{\partial \vu_j(t,a_s)}{\partial a_i^s} )_{j=1}^M$. 
Define the integral equation 
\begin{multline} \label{eq:integral-eqz}
  \vz(t,a^s) = V^s(t,t_0)e_i\\
  + \int_{t_0}^t V^s(t,\tau) D_x\tilde F(\mathbf{u}(\tau,a^s),\tau) \vz(\tau,a) \,d\tau\\ - \int_t^\infty V^u(t,\tau) D_x\tilde F(\mathbf{u}(\tau,a^s),\tau) \vz(\tau,a) \,d\tau.
\end{multline}
It will be shown that \eqref{eq:integral-eqz} yields the desired vector of partial derivatives, i.e., $\vz(t,a_i^s) = (\frac{\partial \vu_j(t,a_s)}{\partial a_i^s} )_{j=1}^M$. Equation \eqref{eq:integral-eqz} may be understood intuitively as follows: Consider taking partial derivatives with respect to $a_i^s$ in \eqref{eq:integral-eq0}. Using the chain rule we see that this is equivalent to \eqref{eq:integral-eqz} if $\vz$ takes the desired form. Equation \eqref{eq:integral-eqz} provides a convenient contractive formula for iteratively approximating $(\frac{\partial \vu_j(t,a_s)}{\partial a_i^s} )_{j=1}^M$. 

Note that, since $\vu(\tau,a^s
)\to 0$ as $\tau\to\infty$, using \eqref{eq:Lip-F} we see that $\|D_x\tilde F(\mathbf{u}(\tau,a^s),\tau)\|$ may be taken to be arbitrarily small by taking $\tau\to\infty$. Again using standard successive approximation techniques (see \cite{coddington1955theory}), we see that there exists a unique solution (in the class of continuous functions) to \eqref{eq:integral-eqz} for all $a^s$ sufficiently small and $t_0$ sufficiently large, and moreover, the solution satisfies
\begin{equation} \label{eq:z-to-zero}
|\vz(t,a^s)| \leq 2K|a^s|e^{-\nu (t-t_0)}
\end{equation}
for $t\geq t_0$ where $\nu$ is as selected after \eqref{eq:V-def}.
We now confirm that $\vz(t,a^s)$ of \eqref{eq:integral-eqz} is in fact equal to $(\frac{\partial \vu_j(t,a^s)}{\partial a^s_{i}})_{j=1}^{M}$. This will be accomplished using standard techniques (see, e.g., \cite{coddington1955theory} Ch. 13).
Let $a\in \R^{n_s}$ and $h>0$. 
$$
\vq(t,a,h) \coloneqq \frac{1}{h}(\vu(t,a+he_i) - \vu(t,a)).
$$
Using \eqref{eq:integral-eq0} we have
\begin{multline} \label{eq:lemma1-eq-q}
\vq(t,a^s,h)=
V^s(t,t_0)e_i \\
+ \int_{t_0}^t V^s(t,\tau)[ D_x \tilde F(\vu(\tau,a^s),\tau)\vq(t,a^s,h) + \Delta]\\
- \int_{t_0}^t V^u(t,\tau)[ D_x \tilde F(\vu(\tau,a^s),\tau)\vq(t,a^s,h) + \Delta]
\end{multline}
where $\Delta = \frac{1}{h}\left[ \tilde F(\vu(\tau,a^s+he_i),\tau) - \tilde F(\vu(\tau,a^s),\tau) \right] - D_x \tilde F(\vu(\tau,a^s),\tau)\vq(t,a^s,h)$.

Let $K$ and $\sigma$ be as in \eqref{eq:ex-estimate-s}.
Using \eqref{eq:Lip-F} we see that for any $\eta>0$ we may choose a sufficiently small neighborhood of the origin such that $|\Delta|<2K\eta$ for all $a^s$ in the neighborhood. Let $\e>0$ be such that $\frac{2K\e}{\sigma} < \frac{1}{2}$.
Using \eqref{eq:lemma1-eq-q} and \eqref{eq:integral-eqz} and letting $m(h) = \sup_{t\geq t_0} \|\vz(t,a^s) - \vq(t,a^s,h)\|$ we have
\begin{align}
m(h) \leq & \, \e\int_{t_0}^t e^{-\sigma(t-\tau)} (m(h) + \|\Delta\|)\dtau\\
& -\e\int_{t}^\infty e^{\sigma(t-\tau)} (m(h) + \|\Delta\|)\dtau\\
 \leq & K\e \eta m(h) \frac{2}{\sigma} + 2K^2\eta \frac{2}{\sigma},
\end{align}
which implies that $m(h) \leq \frac{8K^2\eta}{\sigma}$.
Letting $\eta \to 0$ as $h\to 0$ we see that $m(h)\to 0$ as $h\to 0$, and hence $\vz$ is the desired derivative. 
Finally, the claim that $\frac{\partial \phi_j(t_0,0)}{\partial a_i^s}=0$ follows from \eqref{eq:z-to-zero}.
\end{proof}

\subsection{Proof of Theorem \ref{thrm:non-convergence-saddles}}
Theorem \ref{thrm:non-convergence-saddles} follows readily from Theorem \ref{thrm:main-continuous}. This follows from the fact that under Assumptions
\ref{a:G-connected-undirected} and \ref{a:step-size-CT}--\ref{a:eigvec-continuity}, DGF \eqref{dynamics_CT} is a special case of the general ODE \eqref{eq:ODE1} under Assumptions \ref{a:h-C2}--\ref{a:eigvec-continuity-h} and the assumption that $\gamma_t\to\infty$. Note that the relationship between \eqref{dynamics_CT} and \eqref{eq:ODE1} is made precise in Section \ref{sec:general-setup} (see, in particular, Section \ref{sec:DGD-special-case}). Note also that in this context, if $x^*\in \R^d$ is a saddle point of \eqref{eqn:f-def} and $\tilde x = (x^*,\ldots,x^*)\in\R^{Nd}$ is the $N$-fold repetition, then $\nabla f(x^*)$ corresponds to  $\nabla h\vert_\C(\tilde x)$.\footnote{To simplify notation for the proofs we have set the argument of $h\vert_\C$ to be an element of $\R^{Nd}$. However, modulo this minor abuse of notation, $h\vert_\C$ and \eqref{eqn:f-def} do coincide in this case.}

\subsection{Example and Computation of the Stable Manifold}
In order to illustrate the stable manifolds constructed in the paper, consider an example where 
$$
h(x) = \frac{1}{2}\left(x_1^2 - x_2^2  + x_1^2x_2 + x_1x_2^2\right)(1+x_3) + x_3
$$
and let the constraint space be given by $\calC = \{x\in \R^3:x_3=0\}$. In this example the function has been aligned to the coordinate axis so that $x_3$ plays the role of the off-constraint component while $x_1$ and $x_2$ are the in-constraint components.
A plot of the gradient vector field for $h\vert_{\calC}$ is shown in Figure \ref{fig:vec-field-C}.
\begin{figure}[h] 
\includegraphics[width=8cm]{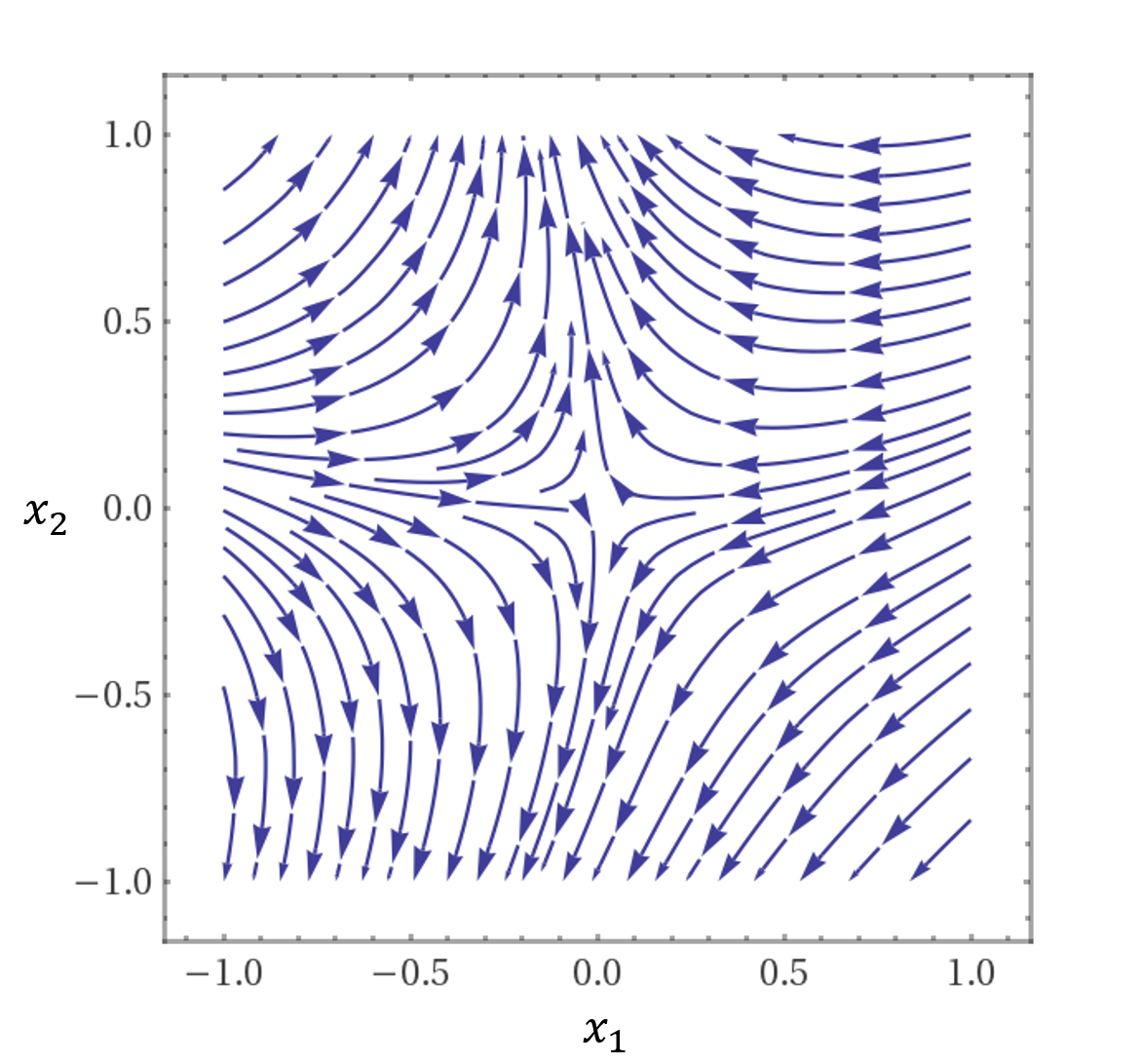}
\caption{}
\label{fig:vec-field-C}
\centering
\end{figure}

The construction of $h$ may be intuitively understood as follows. Observe that the quadratic part of $h$ consists simply of $x_1^2 - x_2^2$. From here, $h$ is constructed by adding higher order terms in order to ``bend'' the stable (and unstable) manifolds, then multiplying by $(1+x_3)$ to warp the vector field away from the constraint space $\calC$, and then finally adding $x_3$ so that $(0,0,0)$ is not an equilibrium of the unconstrained system. 

A plot of  $\calS_{t_0}$ with $t_0=1$ is shown in Figure \ref{fig:stable-man1}.
The stable manifold in Figure \ref{fig:stable-man1} was computed via a Picard-type iteration, i.e., iteratively evaluating the integral equation \eqref{eq:integral-eq0} to obtain the solution $t\mapsto \vu(t,a_s)$ for various values of stable coordinates pairs $a_s = (x_1,x_3)$. The stable manifold is then computed using \eqref{eq:psi-S-def}. As a matter of practical consideration, note that the integral equation \eqref{eq:integral-eq0} is defined with respect to a coordinate change. Care must be taken to ensure that the inputs and outputs of this computation respect this coordinate change. 

We note that here we chose a simple example where $A(t)$, defined in \eqref{eq:A-matrix-def}, is always diagonal, so that the rotation matrix $U(t)$ defined in \eqref{def:Lambda-t} is always the identity. Consequently, this example has no rotational component, and hence the stable manifold here does not exhibit any ``twisting'' behavior as $x_3$ departs from 0. However, twisting behavior can occur in more general examples (particularly, when $U(t)\not= I$). 

It is also worth noting that $\calS_{t_0}$, $t_0=1$ is similar to the stable manifold for two dimensional $\calC$-constrained system visualized in Figure \ref{fig:vec-field-C}, but extrapolated into the $x_3$ dimension. This relationship is not exact, but it can be shown that $\calS_{t_0}$ does converge to an extrapolation of the two-dimensional manifold as $t_0\to \infty$ (see \cite{SMKP2020saddles}, Section 7).

\begin{figure}[h] 
\hspace{.5em}
\includegraphics[width=8cm]{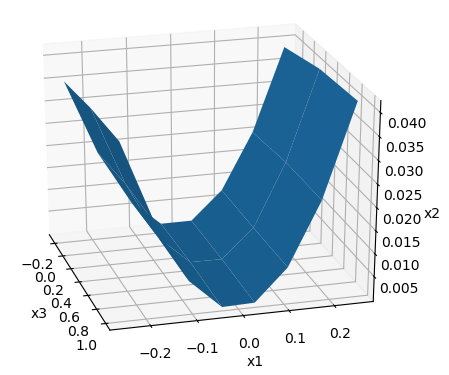}
\caption{}
\label{fig:stable-man1}
\centering
\end{figure}

\section{Conclusions} \label{sec:conclusions}
The paper considered DGF, a multi-agent algorithm for optimizing a distributed sum-over-agents objective. The paper studied convergence to critical points when objectives are permitted to be nonconvex and nonsmooth. In order to make sure that DGF is well-defined in this setting we assume that objectives are Lipschitz continuous and we defined DGF with respect to the generalized gradient. The paper also considered the problem of showing nonconvergence to saddle points In DGF. To handle this problem, the paper assumed that that functions are locally smooth near the saddle and proved the existence of a stable manifold for DGF. We then concluded a.s. convergence to local minima when all saddle points are regular. This paper has focused on continuous-time methods. Discrete-time (stochastic) DGD is treated in the companion paper \cite{SMKP2020saddles}.

\appendix
\section{}

\begin{lemma} [$\calS$ contains all stable initializations] \label{lemma:stable-in-S}
Let $\e$, $r$, and $T$ be chosen as in the construction of $\calS$. Let $a^s\in \R^K$, with $|a^s|< r/3$, let $t_0\geq T$ and suppose that $\vz$ is a solution to \eqref{eq:ODE-z} with $\vz_i(t_0,a^s) = a_i^s$, $i=1,\ldots,k$. If $\vz(t,a^s)\to 0$ as $t\to\infty$ then $(t_0,y_0)\in \calS$.
\end{lemma}
\begin{proof}
By variation of constants we see that
\begin{align} \label{eq:integral-eq}
\vz(t) \coloneqq & V^s(t,t_0)\vz(t_0) + V^u(t,t_0)c\\
& +  \int_{t_0}^t V^s(t,\tau)\left(\tilde F(\vz(\tau),\tau)  -U(\tau)g'(\tau)\dot\gamma_\tau \right)\,d\tau\\
& - \int_{t}^\infty V^u(t,\tau)\left(\tilde F(\vz(\tau)) -U(\tau)g'(\tau)\dot\gamma_\tau \right)\,d\tau,
\end{align}
where $c = \vz(t_0) + \int_{t_0}^\infty V^u(t_0,\tau)\left( \tilde F(\vz(\tau)) -U(\tau)g'(\tau)\dot\gamma_\tau \right)\dtau$. Note that integral in $c$ converges by \eqref{eq:V-def} and the fact that $\int_{t_0}^\infty U(\tau)g'(\tau)\dot\gamma_\tau\dtau < \infty$.
Every term on the right hand side of \eqref{eq:integral-eq} is uniformly bounded in $t$, except possibly the term $V^u(t,t_0)c$. In particular, if $c_j\not = 0$, $j> k$, then $|V^u(t,t_0)c|\to\infty$.
Since the left hand side of \eqref{eq:integral-eq} is bounded uniformly in time, it follows that the right hand side is likewise bounded and thus all $c_j$, $j>k$ must be zero and hence $V^u(t,t_0)c = 0$.

This implies that $\vu(\cdot,a^s) = \vz$ is a solution to the integral equation \eqref{eq:integral-eq0} given $a^s$.
In the proof of Lemma \ref{eq:integral-eq0} we saw that
$\vu(t,a^s)$ is the unique continuous solution of \eqref{eq:integral-eq0} given $a^s$. By the definitions of $\calS$ and $\psi$ we thus see that $(t_0,z_0)\in \calS$.
\end{proof}

The following lemma characterizes the asymptotic properties of the linearization of \eqref{eq:ODE1} near saddle points.
\begin{lemma} \label{lemma:eigvalue-convergence}
Let $A(t)$ be given by \eqref{eq:A-matrix-def} and let $B$ be given by \eqref{def:B}
Let $\{\lambda_1(t),\ldots,\lambda_M(t)\}$ and $\{\lambda_1,\ldots,\lambda_d\}$ denote the eigenvalues of $A(t)$ and $B$ respectively, and assume that $\lambda_i(t) \leq \lambda_j(t)$, $i< j$, and likewise for 
$\lambda_i$, $i=1,\ldots,d$. 
Then $\lambda_i(t)\to\lambda_i$, $i=1,\ldots,d$, and $\lambda_i(t)\to-\infty$, $i=d+1,\ldots,M$.
\end{lemma}

\begin{proof}
This follows by the continuity of eigenvalues as a function of matrices which holds under Assumptions \ref{a:h-C2} and \ref{a:eigvec-continuity-h} (see e.g. \cite{katoBook}, p. 110).
\end{proof}

\bibliographystyle{IEEEtran}
\bibliography{myRefs}

\end{document}